\documentclass[a4paper,11pt]{article}
\usepackage[T1]{fontenc}
\usepackage[utf8]{inputenc}
\usepackage[english]{babel}
\usepackage[margin=0.8in]{geometry}
\usepackage{microtype}
\usepackage{braket}
\usepackage{amsmath}
\usepackage{amssymb}
\usepackage{amsthm, aliascnt}
\usepackage{mathtools}
\usepackage{mathrsfs}
\usepackage{xcolor}
\usepackage[hyphens]{url}
\usepackage{hyperref}
\usepackage{enumitem}
\usepackage{tikz}
\usepackage[maxbibnames=99]{biblatex}
\usepackage{csquotes}
\usetikzlibrary{patterns}
\usepackage{tikz}
\usepackage{esint}

\hypersetup{
                colorlinks=true,
                linkcolor={magenta},
                citecolor={cyan},
                breaklinks=true}

\theoremstyle{plain}
\newtheorem{teor}{Theorem}
\numberwithin{teor}{section}
\numberwithin{equation}{section}
\theoremstyle{definition}
\newaliascnt{defi}{teor}
\newtheorem{defi}[defi]{Definition}
\aliascntresetthe{defi}

\theoremstyle{plain}
\newaliascnt{lemma}{teor}%
\newtheorem{lemma}[lemma]{Lemma}
\aliascntresetthe{lemma}
\theoremstyle{plain}
\newaliascnt{prop}{teor}%
\newtheorem{prop}[prop]{Proposition}
\aliascntresetthe{prop}
\theoremstyle{plain}
\newaliascnt{cor}{teor}%
\newtheorem{cor}[cor]{Corollary}
\aliascntresetthe{cor}
\theoremstyle{definition}
\newaliascnt{ex}{teor}%

\aliascntresetthe{ex}
\theoremstyle{definition}
\newaliascnt{oss}{teor}%
\newtheorem{oss}[oss]{Remark}
\aliascntresetthe{oss}
\theoremstyle{plain}

\DeclarePairedDelimiter{\abs}{\lvert}{\rvert}
\DeclarePairedDelimiter{\norma}{\lVert}{\rVert}

\newcommand{\R}{\mathbb{R}}

\newcommand{\Ln}{\mathcal{L}^n}
\newcommand{\Hn}{\mathcal{H}^{n-1}}

\newcommand{\eps}{\varepsilon}

\newcommand{\ssubset}{\subset\joinrel\subset}
\DeclareMathOperator{\divv}{div}
\DeclareMathOperator{\loc}{loc}
\DeclareMathOperator{\supp}{supp}

\DeclareMathOperator{\Tr}{Tr}
\makeatletter
\newcommand{\leqnomode}{\tagsleft@true\let\veqno\@@leqno}
\newcommand{\reqnomode}{\tagsleft@false\let\veqno\@@eqno}
\newcommand{\gradh}{\frac{\nabla h}{\abs{\nabla h}}}

\DeclareMathOperator{\id}{Id}
\title{On the asymptotic behavior of a diffraction problem with a thin layer}
\author{P. Acampora, E. Cristoforoni}
\date{}

\newcommand{\Addresses}{{%
 \bigskip 
 \footnotesize 
 
 \textsc{Dipartimento di Matematica e Applicazioni ``R. Caccioppoli'', Universit\`a degli studi di Napoli Federico II, Via Cintia, Complesso Universitario Monte S. Angelo, 80126 Napoli, Italy.}\par\nopagebreak 
 
 \medskip 
 
 \textit{E-mail address}, P.~Acampora: \texttt{paolo.acampora@unina.it} 
  
 \medskip

\textsc{Mathematical and Physical Sciences for Advanced Materials and Technologies, Scuola Superiore Meridionale, Largo San Marcellino 10, 80126, Napoli, Italy.}\par\nopagebreak 
 
 \medskip 
 
 \textit{E-mail address}, E.~Cristoforoni: \texttt{emanuele.cristoforoni@unina.it} 
}}
\begin{document}

\maketitle
\begin{abstract}
    We investigate the behavior of the solution to an elliptic diffraction problem in the union of a smooth set $\Omega$ and a thin layer $\Sigma$ locally described by $\eps h$, where $h$ is a positive function defined on the boundary $\partial\Omega$, and $\eps$ is the ellipticity constant of the differential operator in the thin layer $\Sigma$. We study the problem in the limit for $\eps$ going to zero and prove a first-order asymptotic development by $\Gamma$-convergence of the associated energy functional.


    
\end{abstract}
\section{Introduction}
\subsection{The general idea}
Let $\Omega\subset\R^n$ be a bounded open set with $C^{1,1}$ boundary, let $\nu_0$ denote the outer unit normal to $\partial\Omega$, and let $h\in C^{1,1}(\partial\Omega)$. For every $\eps>0$, consider the set 
\[
\Sigma_\eps=\Set{\sigma+t\nu_0(\sigma)|\begin{gathered}\sigma\in\partial\Omega,\\ 0<t<\eps h(\sigma)\end{gathered}}.
\]

\begin{figure}
\caption{Body $\Omega$ surrounded by a thin layer $\Sigma_\eps$}
\centering
\begin{tikzpicture}
    \node[anchor=south west,inner sep=0,opacity=1] (image) at (0,0) {
{\includegraphics[width=.8\textwidth]{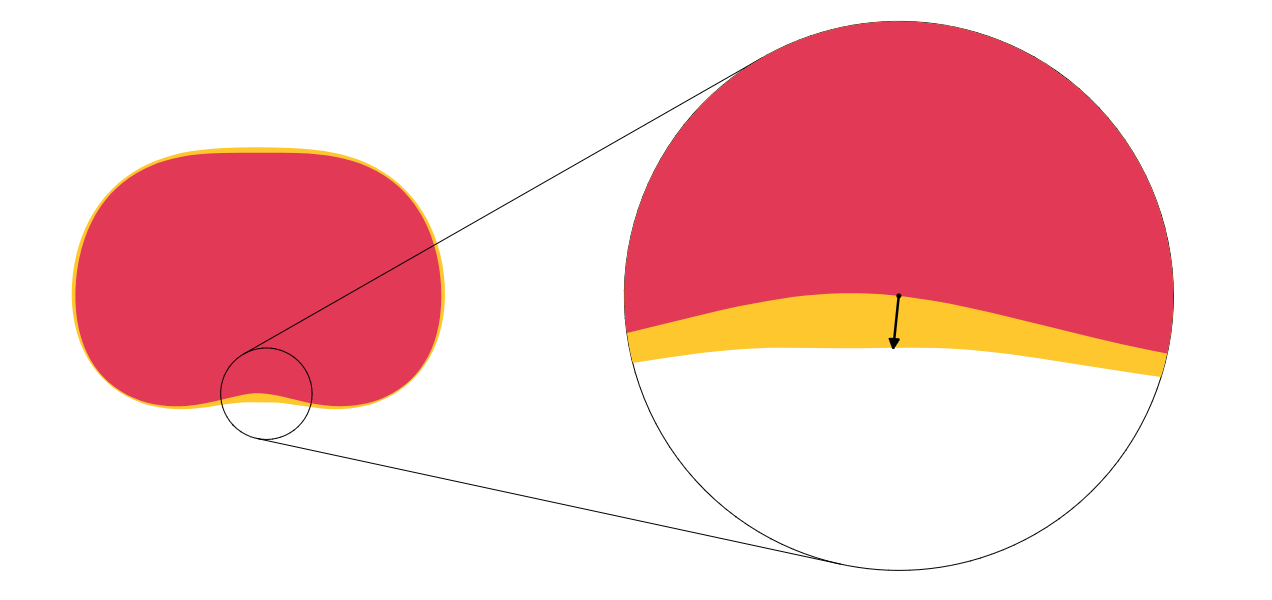}}};
    \begin{scope}[x={(image.south east)},y={(image.north west)}]
        \node[anchor=south west] at (0.18, 0.5) {\Large$\Omega$};
        \node[anchor=south west] at (0.6,0.28) {\Large $\Sigma_\eps$};
        \node[anchor=south west] at (0.62,0.42) {\small $\eps h(\sigma)$};
        \node[anchor=south west] at (0.7,0.5) {$\sigma$};
    \end{scope}
\end{tikzpicture}
\end{figure}
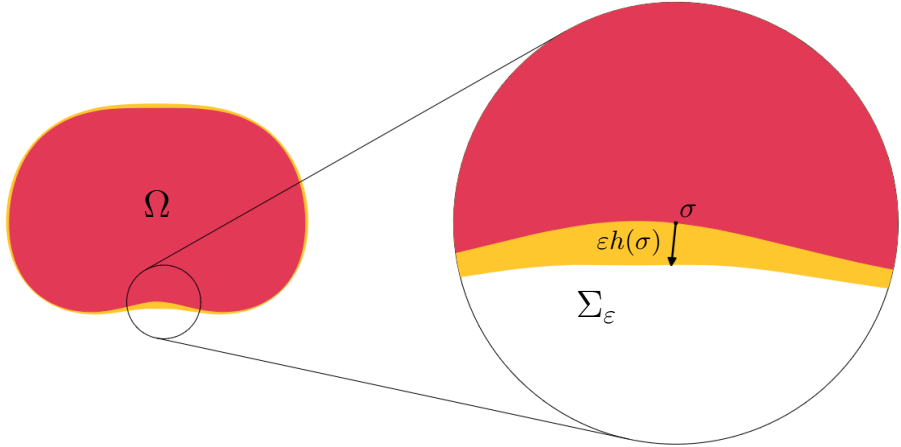

Let $\Omega_\eps=\overline{\Omega}\cup\Sigma_\eps$ and let $\nu_\eps$ be the unit outer normal to $\partial\Omega_\eps$. In this paper, we study the asymptotic behavior, as $\eps$ goes to zero, of the solutions $u_\eps$ to the following boundary value problem:
\begin{equation}
\label{problema}
\begin{cases}-\Delta u_\eps= f & \text{in }\Omega,\\[3 pt]
u_\eps^-=u_\eps^+  & \text{on }\partial\Omega, \\[3 pt]
\dfrac{\partial u_\eps^-}{\partial \nu_0\hphantom{\scriptstyle{-}}}=\eps\dfrac{\partial u_\eps^+}{\partial \nu_0\hphantom{\scriptstyle{+}}} & \text{on }\partial\Omega, \\[6 pt]
\Delta u_\eps=0 & \text{in } \Sigma_\eps, \\[3 pt]
\eps\dfrac{\partial u_\eps}{\partial \nu_\eps} +\beta u_\eps=0 & \text{on } \partial \Omega_\eps,
\end{cases}
\end{equation}
where $f\in L^2(\Omega)$ and $u_\eps^-$ and $u_\eps^+$ denote the traces of $u_\eps$ on $\partial\Omega$ in $\Omega$ and in $\Sigma_\eps$ respectively. As shown in \cite{depiniscatro}, in $\Omega$, the functions $u_\eps$ converges weakly in $H^1(\Omega)$ to $u_0$, the weak solution to the boundary value problem
\begin{equation}
\label{problema0}
\begin{cases}
    -\Delta {u}_0=f &\text{in }\Omega,\\[5 pt]
    \dfrac{\partial {u}_0}{\partial \nu_0}+\dfrac{\beta}{1+\beta h} {u}_0=0 &\text{on }\partial\Omega,
\end{cases}
\end{equation}
i.e. for every $\varphi\in H^1(\Omega)$
\[\int_\Omega \nabla u_0\cdot\nabla\varphi\,dx+\beta\int_{\partial\Omega} \dfrac{u_0\varphi}{1+\beta h}\,d\Hn=\int_\Omega f\varphi\,dx.\]
Let us notice that the functions $u_\eps$ are the minimizers, in $H^1(\Omega_\eps)$, of the functionals 
\[
\mathcal{F}_\eps(v,h)=
    \displaystyle \int_{\Omega}\abs{\nabla v}^2\,dx+\eps\int_{\Sigma_\eps} \abs{\nabla v}^2\,dx+\beta\int_{\partial\Omega_\eps} v^2\,d\Hn-2\int_{\Omega}fv\,dx,
\]
where $v\in H_1(\Omega_\eps)$. 

Our interest in this work is focused on the behavior of solutions $u_\eps$ in the set $\Sigma_\eps$. This interest is mainly motivated by the following remark:  it seems that the solutions to some variational problems in thin sets of the type $\Sigma_\eps$ (see for instance \cite{acrbibuttazzo, depiniscatro,optimalthin}),  can be approximated with functions that are linear along the normal rays to $\partial\Omega$. We will show that, indeed, if we stretch the set $\Sigma_\eps$ to the set $\Sigma_1$, then the stretched solutions $\tilde{u}_\eps$ converge to a function that, in $\Sigma_1$, is linear with respect to the distance from $\partial\Omega$.

In addition, we will show that this convergence property will be sufficient to study the first-order development by $\Gamma$-convergence of the functional $\mathcal{F}_\eps$.

\subsection{The main theorems}
Before stating the main result, we need to fix some notations. Let 
\[
d(x)=d(x,\Omega)=\min_{y\in\bar\Omega}\abs{x-y}
\]
be the distance function from $\Omega$ and define 
\[
\Gamma_{t}=\Set{x\in\R^n|\, d(x)<t}\setminus\Omega.
\]
The regularity of $\Omega$ ensures that $\Omega$ satisfies a uniform exterior ball condition, and in particular, there exists $d_0>0$ such that $d\in C^{1,1}(\Gamma_{d_0})$. Moreover, for every $z\in \Gamma_{d_0}$ there exists a unique $\sigma(z)\in\partial\Omega$ such that
\[
\abs{\sigma(z)-z}=d(z),
\]
and we will say that $\sigma(z)$ is the \emph{metric projection} of $z$ over $\Omega$; we will also denote by $\nu_0$ the extension of the outer unit normal to $\partial\Omega$ such that for every $z\in\Gamma_{d_0}$ we have $\nu_0(z)=\nu_0(\sigma(z))$. 
In the following it will be useful to have that for every $x\in\Sigma_\eps$ the metric projection of $x$ over $\Omega$ exists and it will be useful to work with an extension of $h$ in the set $\Sigma_\eps$ which is constant along the normal directions to $\partial\Omega$. Hence, we will assume $h$ to be a positive $C^{1,1}(\Gamma_{d_0})$ function such that $\norma{h}_\infty<d_0$ and $\nabla h\cdot\nu_0=0$.
To inspect the properties of the solution $u_\eps$ inside $\Sigma_\eps$, we "stretch" the solution $u_\eps$ via a pullback on the reference set $\Sigma_1$. To be more precise, we construct a diffeomorphism
\[
\Psi_\eps:\Sigma_1 \to \Sigma_\eps
\]
by rescaling the distance from $\partial\Omega$. The approach is analogous to the dilation technique proposed in \cite{BL96} for the asymptotic expansion and the derivation of the so-called effective boundary conditions.
We show the following
\begin{teor}\label{teor: tildeu}
    Let $\Omega\subset\R^n$ be a bounded, open set with $C^{1,1}$ boundary, and  fix a positive function $h\in C^{1,1}(\Gamma_{d_0})$ such that $\nabla h\cdot \nu_0=0$, and let $\Psi_\eps$ be the stretching diffemorphism defined in \autoref{def: stretching}. Let $u_\eps$ be the unique weak solution to \eqref{problema}, and let $u_0$ be the unique weak solution to \eqref{problema0}. If we let 
    \[
    \tilde{u}_\eps(z)=\begin{cases}
        u_\eps(z) &\text{if }z\in\Omega,\\[5 pt]
        u_\eps(\Psi_\eps(z)) &\text{if }z\in\Sigma_1,
    \end{cases}
    \] 
    then the family $\tilde{u}_\eps$ is equibounded in $H^1(\Omega_1)$ and, up to a subsequence, converges weakly in $H^1(\Omega_1)$ to the function 
    \[\tilde{u}_0(z)=\begin{cases} 
u_0(z) &\text{if }z\in\Omega,\\[5 pt]
    u_0(\sigma(z))\left(1-\dfrac{\beta d(z)}{1+\beta h(z)}\right) &\text{if }z\in\Sigma_1.
\end{cases}\] 
\end{teor}
\noindent Notice that the limit function $\tilde{u}_0$ turns out to be linear with respect to the distance $d(z)$. 
 \medskip

The functions $u_\eps$ are the minimizers, in $H^1(\Omega_\eps)$, of the functionals 
\[\mathcal{F}_\eps(v,h)=
    \displaystyle \int_{\Omega}\abs{\nabla v}^2\,dx+\eps\int_{\Sigma_\eps} \abs{\nabla v}^2\,dx+\beta\int_{\partial\Omega_\eps} v^2\,d\Hn-2\int_{\Omega}fv\,dx,\]
related to the thermal insulation of a solid, represented by the set $\Omega$, with a thin layer of highly insulating material, represented by the set $\Sigma_\eps$. In this setting, $\eps$ is related to both the mean thickness and the conductivity of the insulating layer $\Sigma_\eps$, while the function  $h$ represents the shape of the layer. Finally, the Robin boundary condition models the case in which the heat exchange with the environment occurs through convection, and $u_\eps$ is the temperature of the insulated body. 

Similar problems have been studied before in the context of reinforcement problems in \cite{BCF}, \cite{friedman}, \cite{acrbibuttazzo}, and more recently in the context of thermal insulation in \cite{bubuni}, \cite{depiniscatro} and \cite{optimalthin}.

In particular, in \cite{depiniscatro} it was proved that the family of functionals $\mathcal{F}_\eps(\cdot,h)$ $\Gamma$-converges, in the strong $L^2$-topology, to the functional
\[ 
\mathcal{F}_0(v,h)=\displaystyle \int_{\Omega}\abs{\nabla v}^2\,dx+\beta\int_{\partial\Omega} \dfrac{v^2}{1+\beta h}\,d\Hn-2\int_{\Omega}fv\,dx.
\] 
Moreover, minimizing the approximated functional $\mathcal{F}_0$ with respect to both variables, the authors proved that, under an integral constraint for the function $h$,  a minimizing couple $(u_0,h_{opt})$ exists and that the optimal function $h_{opt}$, and hence the shape of the insulating layer, should be related to the limit temperature $u_0$, while the dependence on the geometry of $\Omega$ is only implicit. In \cite{optimalthin} a similar problem was studied through a first-order asymptotic development by $\Gamma$-convergence (see \autoref{def:asymptotic}), the main result is in an approximated energy functional whose minimizer $h_{opt}$ explicitly depends on the mean curvature of $\partial\Omega$.

In the spirit of \cite{optimalthin}, we aim to study a first-order asymptotic development by $\Gamma$-convergence of $\mathcal{F}_\eps$. In particular, thanks to \autoref{teor: tildeu}, we are able to prove the following
\begin{teor}
\label{teorema1} 
Let $\Omega\subset\R^n$ be a bounded, open set with $C^{1,1}$ boundary, and  fix a positive function $h\in C^{1,1}(\Gamma_{d_0})$ such that $\nabla h\cdot \nu_0=0$. Then the functional
\[
\delta\mathcal{F}_\eps(\cdot, h)=\dfrac{\mathcal{F}_\eps(\cdot,h)-\mathcal{F}_0(u_0,h)}{\eps}
\]
$\Gamma$-converges, in the strong $L^2(\R^n)$ topology, as $\eps\to 0^+$, to
\[
    \mathcal{F}^{(1)}(v,h)=%
    \begin{cases}%
        \displaystyle \beta\int_{\partial\Omega} \frac{Hh(2+\beta h)}{2(1+\beta h)^2}\,u_0^2\,d\Hn, %
        &\text{if }v=u_0,\\[10 pt] %
        +\infty &\text{if }v\ne u_0,
    \end{cases}
\]
where $H$ denotes the mean curvature of $\partial\Omega$.
\end{teor}

Let us look at the first-order approximation in $\eps$ of the energy $\mathcal{F}_\eps(u_\eps,h)$, described by
\[\mathcal{G}_\eps(h)=\mathcal{F}_0(u_0,h)+\eps\mathcal{F}^{(1)}(u_0,h).\]
We believe that the study of $\mathcal{G}_\eps$,  either with a volume penalization as in \cite{CK,BL14,ACNT23} or a volume constraint as in \cite{bubuni,depiniscatro}, will provide a deeper understanding of the optimal shape of the insulating layer. From the previous result, we conjecture that the design of the shape of the optimal insulating layer should not only be related to the limit temperature $u_0$, but also to the mean curvature $H$. However we were not able to provide an existence result for minimizers of the approximating energy.

The present paper is organized as follows. In \autoref{notations} we give introductory definitions and tools about hypersurfaces and $\Gamma$-convergence. In \autoref{stretching} we explicitly compute the diffeomorphism $\Psi_\eps$ and we prove \autoref{teor: tildeu} using some energy estimates that will be proved later in the paper. In \autoref{Gamma}, using the convergence of the functions $\tilde{u}_\eps$, we explicitly compute the first-order asymptotic development by $\Gamma$-convergence for the functional $\mathcal{F}_\eps(\cdot,h)$. Finally, in \autoref{EE} we prove the aforementioned energy estimates, using classical techniques in regularity theory adapted to this particular case of a diffraction problem.
\section{Notations and tools}
\label{notations}
    
\subsection{Distance function and curvatures}
\label{sec: distcurv}
We refer to \cite[Section 14.6]{trudinger} for the notions in this section. Let $d$, $\Gamma_t$, $\sigma$ and $\nu_0$ as in the introduction. For every $x\in\Gamma_{d_0}$ we can write
\[x=\sigma(x)+d(x)\nu_0(x),\] 
and notice that we have $\sigma,\nu_0\in C^{0,1}(\Gamma_{d_0};\R^n)$, and $\nabla d=\nu_0$. For $\Hn$-a.e. $\sigma\in\partial\Omega$ we can consider the Hessian matrix $D^2 d(\sigma)$, and, since $\abs{\nabla d}=1$, we have that $\nu_0(\sigma)$ is an eigenvector with corresponding zero eigenvalue. Let $\set{\tau_1,\dots,\tau_{n-1},\nu_0}$ be a system of normalized eigenvectors.
\begin{defi}
Let $\sigma\in \partial\Omega$ and let $\set{\tau_1(\sigma),\dots,\tau_{n-1}(\sigma),\nu_0(\sigma)}$ be an ordered system of normalized eigenvectors for $D^2d(\sigma)$. We define the \emph{principal curvatures}  $k_1(\sigma),\dots,k_{n-1}(\sigma)$ of $\partial\Omega$ at the point $\sigma$ as the eigenvalues of the matrix
\[
D^2d(\sigma)
\] 
corresponding to the eigenvectors $\tau_1(\sigma),\dots,\tau_{n-1}(\sigma)$.
\end{defi}

Let $t\in(0,d_0)$ and consider 
\[
\gamma_t=\partial(\Omega\cup\Gamma_t)=\Set{x\in\R^n|d(x)=t}\setminus\Omega.
\]
Our regularity assumptions allow us to consider for every $x\in\Gamma_{d_0}$ the matrix $D^2d(x)$, which is symmetric and represents the second fundamental form of $\gamma_{d(x)}$. Moreover, by direct computation  (see for instance \cite[Lemma 14.17]{trudinger}) we have that $\tau_i(\sigma(x))$ are eigenvectors for $D^2d(x)$ with eigenvalues computed in the following definition.
\begin{defi}
    \label{def: curv}
    Let $x\in\Gamma_{d_0}$. For every $i=1,\dots,n-1$ we denote by
    \[
        \tau_i(x):=\tau_i(\sigma(x)),
    \]
    and we denote by $k_1(x),\dots,k_{n-1}(x)$ the corresponding sequence of eigenvalues of $D^2d(x)$
    \[
        k_i(x):=\dfrac{k_i(\sigma(x))}{1+d(x) k_i(\sigma(x))},
    \]
\end{defi}
\begin{oss}
\label{oss: gammacurvatures}
By the properties of $D^2d(x)$ we can diagonalize $D\sigma(x)$ as
\[
D\sigma(x) \,\tau_i(\sigma)=\frac{1}{1+d(x)k_i(\sigma)}\tau_i(\sigma).
\]
\end{oss}

\subsection{Calculus on hypersurfaces}
We refer to \cite{maggi} for the notions in this section. For every $v,w\in \R^n$ let the \emph{tensorial product} $v\otimes w$ be the unique linear operator on $\R^n$ such that, for every $z\in\R^n$,
\[
(v\otimes w)z=(w\cdot z) v.
\]

Let $\Omega\subset\R^n$ be a bounded open set with $C^1$ boundary and let $\nu_0$ be the outer unit normal to its boundary. For every $\sigma\in\partial\Omega$ we denote by
\begin{gather*}
T_\sigma\partial\Omega=\Set{y\in\R^n | \nu_0(\sigma)\cdot y=0},
\end{gather*}
the \emph{tangent space to $\partial\Omega$ at $\sigma$}. Let $\tau_1(\sigma),\dots,\tau_{n-1}(\sigma)$ be an orthonormal basis for $T_\sigma \partial\Omega$.

\begin{defi}[Tangential gradient]
Let $\Omega$ be a bounded open set of class $C^1$, let $U\subseteq\R^n$ be an open set containing $\partial\Omega$, and let $\phi\in C^{0,1}(U;\R^n)$. We define for $\Hn$-a.e. $\sigma\in\partial\Omega$ the \emph{tangential gradient of $\phi$} at $\sigma$ as the the linear map 
\[
D^{\partial\Omega} \phi(\sigma)\colon T_\sigma \partial\Omega \to\R^{n}
\]
defined as
\[\begin{split}D^{\partial\Omega} \phi(\sigma)=&\sum_{h=1}^{n-1}\left(\nabla \phi(\sigma)\tau_h\right)\otimes\tau_h\\[5 pt]
=&\nabla\phi(\sigma)-\left(\nabla \phi(\sigma)\nu_0\right)\otimes\nu_0.
\end{split}\]
\end{defi}
Notice that the definition of tangential gradient does not depend on the choice of the orthonormal basis $\tau_1,\dots,\tau_{n-1}$.
\begin{defi}[Tangential Jacobian]
Let $\Omega$ be a bounded open set of class $C^1$, let $U\subseteq\R^n$ be an open set containing $\partial\Omega$, and let $\phi\in C^{0,1}(U;\R^n)$. We define \emph{the tangential Jacobian of $\phi$} as 
\[
J^{\partial\Omega} \phi =\sqrt{\det\big((D^{\partial\Omega} \phi)^T (D^{\partial\Omega} \phi)\big)\,},\,
\]
where the determinant has to be intended in the space $T_\sigma \partial\Omega \otimes T_\sigma\partial\Omega$.
\end{defi}
\begin{teor}[Area formula on surfaces]\label{teor:area}
Let $U\subseteq\R^n$ be an open set containing $\partial\Omega$, let $\phi\in C^{0,1}(U;\R^n)$, and let $g:\R^n\to\R$ be a positive Borel function. We have that
\[\int_{\partial\Omega} g(\phi(\sigma))\, J^{\partial\Omega} \phi\,d\Hn=\int_{\phi(\partial\Omega)} g(\sigma)\,d\Hn.\]
\end{teor}

\begin{defi}[Tangential divergence]
Let $\Omega$ be a bounded open set of class $C^1$, let $U\subseteq\R^n$ be an open set containing $\partial\Omega$ and let $\phi\in C^{0,1}(U;\R^n)$. We define the \emph{tangential divergence of $\phi$} as 
\[
\divv^{\partial\Omega} \phi = \sum_{j=1}^{n-1} (D \phi \,\tau_j)\cdot\tau_j
\]
\end{defi}
\begin{defi}[Mean Curvature]\label{def:H}
Let $\Omega$ be a bounded open set with $C^{1,1}$ boundary and let 
\[
\nu: \partial\Omega\to\mathbb{S}^{n-1}
\] 
be the outer unit normal to its boundary. Let $U\subseteq\R^n$ be an open set containing $\partial\Omega$ and let $X \in C^{0,1}(U)$ be an extension of $\nu$. For $\Hn$-a.e. $\sigma\in\partial\Omega$ we define the \emph{mean curvature of $\partial\Omega$} as
\[H(\sigma)=\divv^{\partial\Omega} X=\sum_{i=1}^{n-1}k_i(\sigma).\]
\end{defi}

\begin{oss}\label{ossimp}
Let $\Omega$ be a bounded open set of class $C^{1,1}$, let $U\subseteq\R^n$ be an open set containing $\partial\Omega$, let $X\in C^{0,1}(U;\R^n)$, and let $\phi(x)=x+tX(x)$. By direct computations, we have that
\[J^{\partial\Omega} \phi(\sigma)=1+t\divv^{\partial\Omega} X(\sigma)+t^2R(t,\sigma)\]
where the remainder $R$ is a bounded function. In particular, if $X$ is an extension of $\nu$, we have   
\[J^{\partial\Omega} \phi(\sigma)=1+t H(\sigma)+ t^2R(t,\sigma).\]
\end{oss}

\begin{teor}[Coarea formula]\label{coarea}
Let $f\in C^{0,1}(\R^n)$, let $g\in L^1(\R^n)$, and let $U\subset\R^n$ be an open set. Then
\[\int_U g(x)\abs{\nabla f(x)}\,dx=\int_{\R} \int_{U\cap\Set{f=t}} g(y)\,d\Hn(y)\,dt.\]
\end{teor}

\subsection{\texorpdfstring{$\Gamma$}{}-convergence}
In this section, we recall some basic properties of the $\Gamma$-convergence and the asymptotic development by $\Gamma$-convergence. We refer for instance to \cite{dalmaso} and \cite{anzellotti1993asymptotic} for the following notions.
\begin{defi}
Let $X$ be a metric space and, for any $\eps>0$, let us consider the functionals $\mathcal{F}_\eps,\mathcal{F}_0 : X\to \R\cup\set{+\infty}$. We will say that \emph{$\mathcal{F}_\eps$ $\Gamma$-converges, with respect to the strong topology in $X$, as $\eps\to 0^+$} to $\mathcal{F}_0$ if for every $x\in X$ the following conditions hold:
\begin{itemize}
 \item{for every sequence $\set{x_\eps}\subset X$ converging to $x$,
 \[\liminf_{\eps\to0^+}\mathcal{F}_\eps(x_\eps)\geq \mathcal{F}_0(x);\]}
 \item{there exists a sequence $\set{x_\eps}\subset X$ converging to $x$ such that \[\limsup_{\eps\to0^+}\mathcal{F}_\eps(x_\eps)\le \mathcal{F}_0(x).\]}
\end{itemize}
\end{defi}
In particular, from the definition, if $\mathcal{F}_\eps$ $\Gamma$-converges to $\mathcal{F}_0$, for every $x\in X$ there exists a recovery sequence $\set{x_\eps}\subset X$, converging to $x$, such that
\[\lim_{\eps\to0^+} \mathcal{F}_\eps(x_\eps)=\mathcal{F}_0(x).\]
We have the following
\begin{prop}\label{prop:convminimi}
Let $X$ be a metric space and, for any $\eps>0$, let us consider the functionals $\mathcal{F}_\eps,\mathcal{F}_0 : X\to\R\cup\set{+\infty}$ such that $\mathcal{F}_\eps$ $\Gamma$-converges, with respect to the strong topology in $X$ as $\eps\to 0^+$ to $\mathcal{F}_0$. Let $\set{x_\eps}$ be a sequence in $X$ such that \[\mathcal{F}_\eps(x_\eps)=\min_X \mathcal{F}_\eps.\]
If there exists $\overline{x}\in X$ such that $x_\eps$ converges to $\overline{x}$, then
\[\mathcal{F}_0(\overline{x})=\min_X \mathcal{F}_0=\lim_{\eps\to0^+}\min_X \mathcal{F}_\eps.\]
\end{prop}

Let
\[
m_0=\inf_X \mathcal{F}_0,
\]
and, for every $x\in X$, let
\[
\delta\mathcal{F}_\eps(x)=\frac{\mathcal{F}_\eps(x)-m_0}{\eps}
\]
\begin{defi}\label{def:asymptotic}
If there exists a functional $\mathcal{F}^{(1)}\colon X\to \R\cup\set{+\infty}$ such that $\delta\mathcal{F}_\eps$ $\Gamma$-converges, with respect to the strong topology in $X$, as $\eps\to 0^+$ to $\mathcal{F}^{(1)}$, we say that $\mathcal{F}^{(1)}$ is the \emph{first-order asymptotic development by $\Gamma$-convergence} for the functional $\mathcal{F}_\eps$.
\end{defi}

Let
\[\mathcal{U}_0=\Set{x\in X|\mathcal{F}_0(x)=m_0},\]
the interest in the previous definition is justified by the following
\begin{oss}
\label{oss: approx}
Let $\set{x_\eps}$ be a sequence in $X$ such that \[\mathcal{F}_\eps(x_\eps)=\min_X \mathcal{F}_\eps,\]
and assume that there exists $\overline{x}\in X$ such that $x_\eps$ converges to $\overline{x}$; then, by \autoref{prop:convminimi}, we have that $\overline{x}\in\mathcal{U}_0$ and
\[\mathcal{F}^{(1)}(\overline{x})=\min_X \mathcal{F}^{(1)}=\lim_{\eps\to0^+}\dfrac{\mathcal{F}_\eps(x_\eps) - m_0}{\eps}.\]
In particular, we have
\[\mathcal{F}_\eps(x_\eps)=m_0+\eps \mathcal{F}^{(1)}(\bar{x})+o(\eps).\]
\end{oss}
\section{Stretching}
\label{stretching}
In this section, we construct the diffeomorphism 
$\Psi_\eps\colon\Sigma_1\to\Sigma_\eps$, and, assuming that the family of functions $\tilde{u}_\eps=u_\eps\circ\Psi_\eps$ converges weakly in $H^1(\Sigma_1)$ to a function $\tilde{u}_0$, we compute the limiting boundary value problem for the limit $\tilde{u}_0$.

\subsection{Preliminary computations}

Let $\Omega$ be a bounded open set with $C^{1,1}$ boundary. For every $x\in\Gamma_{d_0}$, we recall we have defined
\[
\nu_0(x)=\nu_0(\sigma(x))
\]
to be the unit outer normal to $\partial\Omega$ in $\sigma(x)$, and let $h\in C^{1,1}(\Gamma_{d_0})$ be a positive function such that $\nabla h\cdot \nu_0=0$, so that $h$ is constant along normal radii starting from $\partial\Omega$. We define
\[
\Sigma_\eps :=\Set{x\in \R^n | 0<d(x)<\eps h(x)},
\qquad \qquad 
\Omega_\eps=\overline{\Omega}\cup\Sigma_\eps.
\]
In particular, notice that, if $\eps\norma{h}_\infty<d_0$, we can write
\begin{equation}
\label{eq: Sigmaepsdef}
\Sigma_\eps=\Set{\sigma+t\nu_0(\sigma)|\begin{gathered}
\sigma\in\partial\Omega, \\
0<t<\eps h(\sigma)
\end{gathered}},
\end{equation}
and the representation $x=\sigma+t\nu_0$ is unique and $d\in C^{1,1}(\Sigma_\eps)$. 

For simplicity's sake, we will assume that $\norma{h}_\infty<d_0$ so that we can assume the distance function $d$ to be regular on the set $\Sigma_1$. 
\begin{defi}
\label{def: stretching}
We define the \emph{stretching diffeomorphism} $\Psi_\eps\in C^{0,1}(\Gamma_{d_0};\Gamma_{\eps d_0})$ as the function defined as
\[
\begin{split}
\Psi_\eps(z)&=\sigma(z)+\eps d(z)\nu_0(z) \\[5 pt] 
&=z+(\eps-1)d(z)\nu_0(z).
\end{split}
\]
\end{defi}
With this definition we have that 
\[\begin{split}
D\Psi_\eps(z)&=D\sigma(z)+\eps\left[\nu_0(z)\otimes\nu_0(z)+d(z) D\nu_0(z)\right]\\[5 pt]
&=\id_n+(\eps-1)\left[\nu_0(z)\otimes\nu_0(z)+d(z) D\nu_0(z)\right],
\end{split}
\]
where $\id_n$ is the identity matrix. Moreover, $\Psi_\eps$ is invertible with
 \[
    \Psi_\eps^{-1}(x)=\sigma(x)+\dfrac{d(x)}{\eps}\nu_0(x).
\]
By direct computations, we have the following

\begin{lemma}
\label{lem: areaformulasstretch}
    Let $\Omega\subset\R^n$ be a bounded, open set with $C^{1,1}$ boundary, and  fix a positive function $h\in C^{1,1}(\Gamma_{d_0})$ such that $\nabla h\cdot \nu_0=0$. Let $g:\R^n\to\R$ be a positive Borel function. Then
    \begin{equation}
    \label{eq: changeofvariablesdLn}
        \int_{\Sigma_{\eps}}g(x)\,dx=\eps\int_{\Sigma_1}g(\Psi_\eps(z))J_{\eps}(z)\,dz,
    \end{equation}
    with
    \[
        J_{\eps}(z)=\prod_{i=1}^{n-1}\dfrac{1+\eps d(z) k_i(\sigma)}{1+d(z) k_i(\sigma)},
    \]
    and $\sigma=\sigma(z)$. In addition,
    \begin{equation}
    \label{eq: changeofvariablesdHn}
        \int_{\partial\Omega_{\eps}}g(\xi)\,d\Hn(\xi)=\int_{\partial\Omega_1}g(\Psi_\eps(\zeta))J^\tau_{\eps}(\zeta)\,d\Hn(\zeta),
    \end{equation}
    where $J_\eps^\tau$ is the tangential Jacobian of $\Psi_\eps$, and it converges uniformly, as $\eps\to 0^+$, to
    \[
    \begin{split}
        J^\tau_{0}(\zeta)&=\frac{1}{\sqrt{1+\abs{\nabla h(\zeta)}^2}}\prod_{i=1}^{n-1}\frac{1}{1+h(\zeta)k_i(\sigma(\zeta))
        } \\[7 pt]
        &=\frac{1}{\sqrt{1+\abs{\nabla h(\zeta)}^2}} J_0(\zeta).
    \end{split}
    \]
\end{lemma}
\begin{proof}
    Let $z\in\Gamma_{d_0}$, and let $\tau_i(z)$ be the eigenvectors of $D^2d(z)=D\nu_0(z)$, with respective eigenvectors $k_i(z)$, defined in \autoref{def: curv}. We have that $\set{\tau_1(\sigma),\dots,\tau_{n-1}(\sigma),\nu_0(\sigma)}$ is a basis of eigenvectors for $D\Psi_\eps(z)$, and, in particular,
\[
\begin{split}
D\Psi_\eps(z)\,\tau_i(z)&=\left(1+\dfrac{(\eps-1)d(z)k_i(\sigma)}{1+d(z)k_i(\sigma)}\right)\tau_i(z)\\[5 pt]
&=\dfrac{1+\eps d(z)k_i(\sigma)}{1+d(z)k_i(\sigma)}\tau_i(z),
\end{split}
 \] 
 and
 \[
 D\Psi_\eps(z)\,\nu_0(z)=\eps\nu_0(z).
 \]
Therefore, \eqref{eq: changeofvariablesdLn} follows from the area formula.
\medskip

On the other hand, notice that $D\Psi_\eps$ converges uniformly, as $\eps\to 0^+$, to $D\sigma$,
so that $J^{\partial\Omega_1} \Psi_\eps$ converges to $J^{\partial\Omega_1} \sigma$. Therefore, we apply the area formula on surfaces (\autoref{teor:area}) to get \eqref{eq: changeofvariablesdHn}, and it is left to compute the tangential Jacobian $J^{\partial\Omega_1}\sigma$. 
\medskip

Let $\zeta\in\partial\Omega_1$, and let $\nu_1$ be the unit outer normal to $\partial\Omega_1$. Recalling that $\nabla d=\nu_0$, the definition of $\Sigma_1$ ensures that
\[
\nu_1(\zeta)=\dfrac{\nu_0(\zeta)-\nabla h(\zeta)}{\sqrt{1+\abs{\nabla h(\zeta)}^2}}.
\]

We aim to construct an orthonormal basis for the tangent space $T_\zeta\partial \Omega_1$ with a rotation of the tangent space $T_{\sigma(\zeta)}\partial \Omega$. In the following, when possible, we will drop the dependence on $\zeta$. Let us define the rotation operator
\[
R:\R^n \to \R^n
\]
as the unique linear operator having the following properties:
\begin{enumerate}[label=(\roman*)]
\item if $w\in (\nabla h)^\perp\cap\nu_0^\perp$, then $Rw=w$;
\item $R \nu_0 =\nu_1$;
\item $\nabla h/\abs{\nabla h}$ is mapped to a unit vector laying in the plane generated by $\nabla h$ and $\nu_0$, and orthogonal to $\nu_1$, namely 
\[
R \gradh=\frac{1}{\sqrt{1+\abs{\nabla h}^2}}\biggl(\frac{\nabla h}{\abs{\nabla h}}+ \abs{\nabla h}\nu_0\biggr).
\]
\end{enumerate}
Explicitly, we can write $R$ as
\[
\begin{split}
R=&\id_n-\nu_0\otimes\nu_0-\frac{\nabla h}{\abs{\nabla h}}\otimes\frac{\nabla h}{\abs{\nabla h}}+\nu_1\otimes \nu_0\\[5 pt]
&+\frac{1}{\sqrt{1+\abs{\nabla h}^2}}\left(\frac{\nabla h}{\abs{\nabla h}} +\abs{\nabla h}\nu_0\right)\otimes\frac{\nabla h}{\abs{\nabla h}}.
\end{split}
\]
This operator is a rotation in $\R^n$ that maps $T_{\sigma(\zeta)}\partial\Omega$ onto $T_{\zeta}\partial\Omega_1$, and, in particular, we define an orthonormal basis $\bar\tau_i$ of $T_\zeta\partial\Omega_1$ as follows: let $\tau_i=\tau_i(\sigma(\zeta))$ be an orthonormal basis  of $T_{\sigma(\zeta)}\partial\Omega$, and we define
\[
\bar \tau_s:=R\tau_s=\left(\id_n+\left(\frac{1}{\sqrt{1+\abs{\nabla h}^2}}-1\right)\gradh\otimes\gradh\right)\tau_s+\frac{\nabla h\cdot \tau_s}{\sqrt{1+\abs{\nabla h}^2}}\,\nu_0.
\]
In particular, since $D\sigma \,\nu_0=0$, then
\[
(D\sigma \circ R)\tau_s=D\sigma \left(\id_n+\left(\frac{1}{\sqrt{1+\abs{\nabla h}^2}}-1\right)\gradh\otimes\gradh\right)\tau_s.
\]
Hence, recalling the evaluation of the eigenvalues of $D\sigma$ given in \autoref{oss: gammacurvatures}, we can easily compute the determinant of the tangential gradient
\[
\begin{split}
J^{\partial\Omega_1} \sigma(\zeta)=\left(\prod_{i=1}^{n-1}(1+h(\zeta)k_i(\sigma(\zeta)))\sqrt{1+\abs{\nabla h(\zeta)}^2}\right)^{-1},
\end{split}
\]
thus concluding the proof.

\end{proof}
\subsection{The main result}
Fix $\beta>0$. For every $0<\eps<1$ let $u_\eps\in H^1(\Sigma_\eps)$ be a solution to 
\begin{equation}\label{eq:eqeps}
    \eps\int_{\Sigma_\eps}\nabla u_\eps \nabla\varphi \,dx+\beta\int_{\partial\Omega_\eps}u_\eps\varphi\,d\Hn=0,
\end{equation}
for every $\varphi\in H^1(\Sigma_\eps)$ such that $\varphi=0$ on $\partial\Omega$. Then we have the following.
\begin{prop}
\label{prop: weaklim}
Let $\Omega\subset\R^n$ be a bounded, open set with $C^{1,1}$ boundary, and  fix a positive function $h\in C^{1,1}(\Gamma_{d_0})$ such that $\nabla h\cdot \nu_0=0$. Let $u_\eps$ be a family of weak solution to \eqref{eq:eqeps}, and let $\tilde{u}_{\eps}(z)=u_\eps(\Psi_\eps(z))$. If $\tilde{u}_\eps$ converges weakly in $H^1(\Sigma_1)$ to a function $\tilde{u}_0$, then $\tilde{u}_0$ is a solution to
\begin{equation}\label{eqlimweak}
\int_{\Sigma_1} (\nabla\tilde{u}_0\cdot\nu_0)(\nabla\varphi\cdot\nu_0)J_0(z)\,dz+\beta\int_{\partial\Omega_1} \tilde{u}_0(\zeta)\varphi(\zeta)\dfrac{J_0(\zeta)}{\sqrt{1+\abs{\nabla h(\zeta)}^2}}\,d\Hn(\zeta)=0. \end{equation}
for every $\varphi\in H^1(\Sigma_1)$ such that $\varphi=0$ on $\partial\Omega$.
\end{prop}
\begin{proof}
By definition, we have $\tilde{u}_\eps\in H^1(\Sigma_1)$ and
\begin{equation}
    \label{eq: nablaueps}
\begin{split}
\nabla u_\eps|_{\Psi_\eps(z)}&=D(\Psi_\eps^{-1})|_{\Psi_\eps(z)} \nabla \tilde{u}_\eps(z)\\[10 pt]
&=\left(\dfrac{1}{\eps}\nabla\tilde{u}_\eps(z)\cdot\nu_0(z)\right)\,\nu_0(z)+\sum_{i=1}^{n-1}\left(\dfrac{1+d(z)k_i(\sigma)}{1+\eps d(z)k_i(\sigma)} \nabla \tilde{u}_\eps(z)\cdot\tau_i(z)\right)\,\tau_i(z).
\end{split}
\end{equation}

Let $\varphi\in H^1(\Sigma_1)$ with $\varphi=0$ on $\partial\Omega$, and let $\varphi_\eps(x)=\varphi(\Psi_\eps^{-1}(x))$, then equation \eqref{eq:eqeps} yields
\[
\eps\int_{\Sigma_\eps}\nabla u_\eps \nabla\varphi_\eps \,dx+\beta\int_{\partial\Omega_\eps}u_\eps\varphi_\eps\,d\Hn=0,
\]
from which, using \autoref{lem: areaformulasstretch} and the computation \eqref{eq: nablaueps}, we have
\begin{equation}
\label{eq: weakeqeps}
\begin{split}
\int_{\Sigma_1}\left[ (\nabla\tilde{u}_\eps\cdot\nu_0)(\nabla\varphi\cdot\nu_0)+ \eps^2\sum_{i=1}^{n-1} \left(\dfrac{1+d k_i}{1+\eps d k_i}\right)^2 (\nabla \tilde{u}_\eps\cdot\tau_i)(\nabla\varphi\cdot\tau_i)\right]\prod_{i=1}^{n-1}\dfrac{1+\eps d k_i}{1+d k_i}\,dz+\\[15 pt]
+\beta\int_{\partial\Omega_1}\tilde{u}_\eps\varphi J^\tau_\eps\,d\Hn=0,
\end{split}
\end{equation}
Passing to the limit in \eqref{eq: weakeqeps}, the assertion follows.
\end{proof}

\begin{oss}[Uniqueness]
    For every given Dirichlet boundary condition on $\partial\Omega$, \eqref{eqlimweak} admits a unique solution. Indeed, let $v_1,v_2\in H^1(\Sigma_1)$ be two solutions to \eqref{eqlimweak} such that $v_1=v_2$ on $\partial\Omega$, and let $w=v_1-v_2$. By linearity, we have that $w$ is a solution to \eqref{eqlimweak} with $w=0$ on $\partial\Omega$, so that we have
\[\int_{\Sigma_1} \abs{\nabla w\cdot\nu_0}^2 J_0\,dz+\beta\int_{\partial\Omega_1} w^2\dfrac{J_0}{\sqrt{1+\abs{\nabla h}^2}}\,d\Hn=0,\]
and since $J_0>0$, then $\nabla w\cdot\nu_0=0$ a.e. on $\Sigma_\eps$. Then, for $\Ln$-a.e. $z\in\Sigma_1$,
\[w(z)=w(\sigma(z))+\int_0^{d(z)} \nabla w(\sigma(z)+t\nu_0)\cdot \nu_0\,dt=0,\]
 that is $w=0$ and $v_1=v_2$.
\end{oss}
\begin{oss}[Limit computation]
\label{oss: explsol}
We point out that it is possible to explicitly compute the solution $\tilde{u}_0$ to \eqref{eqlimweak} in terms of its values on $\partial\Omega$ as 
\[
\tilde{u}_0(z)=\tilde{u}_0(\sigma(z))\left(1-\frac{\beta d(z)}{1+\beta h(z)}\right).
\] 
We first inspect the regular case of $\Omega$ of class $C^3$ to obtain the strong equation, and then we work on the general case using only the weak equation \eqref{eqlimweak}.

Indeed, let 
\[A(z)=J_0(z)\nu_0(z)\otimes\nu_0(z),\]
and notice that when $\Omega$ is smooth, then $J_0$ is also smooth as it can be seen as the determinant of the smooth matrix-valued map
\[z\in\Sigma_1\mapsto D\sigma+\nu_0\otimes\nu_0.\]
Then, if $\tilde{u}_0$ is a regular solution to \eqref{eqlimweak}, it is a solution to
\begin{equation}\label{eqlimit0}
\begin{cases}
    \divv\left(A(x)\nabla\tilde{u}_0\right)=0 &\text{in }\Sigma_1,\\[5 pt]
    \dfrac{\partial \tilde{u}_0}{\partial\nu_0} +\beta\tilde{u}_0=0&\text{on }\partial\Omega_1.
\end{cases}
\end{equation}
Moreover, we have that
 \[\divv(J_0\nu_0)=\nabla J_0\cdot \nu_0+J_0\Tr(D\nu_0).\]
In particular, we can explicitly compute the derivative of $J_0$ in direction $\nu_0$ using the local representation
\[J_0(z)=\prod_{i=1}^{n-1}\frac{1}{1+d(z)k_i(\sigma(z))},\]
and recalling that $k_i\circ \sigma$ is constant along normal radii, so that, for every $\zeta\in\partial\Omega_1$,
\[\nabla J_0(\zeta)\cdot \nu_0(\zeta)=-J_0(\zeta)\Tr(D\nu_0(\zeta)).\]
Hence, $\divv(J_0\nu_0)=0$, and
\[\divv(A\nabla u)=J_0\nabla(\nabla u\cdot \nu_0)\cdot\nu_0.\]

Therefore, equation \eqref{eqlimit0} reduces to 
\begin{equation}\label{eqlimit}
\begin{cases}
    \nabla(\nabla \tilde{u}_0\cdot \nu_0)\cdot\nu_0=0 &\text{in }\Sigma_1,\\[5 pt]
    \dfrac{\partial \tilde{u}_0}{\partial\nu_0} +\beta\tilde{u}_0=0&\text{on }\partial\Omega_1.
\end{cases}
\end{equation}
The previous computation suggests that solutions to \eqref{eqlimweak} have to be linear with respect to the normal direction. Indeed, Since $h$ is constant along the normal direction, it is sufficient to check that, for every $w\in H^1(\Sigma_1)$ the function
\[\tilde{u}(z)=w(\sigma(z))\left(1-\dfrac{\beta d(z)}{1+\beta h(z)}\right)\]
is a solution to \eqref{eqlimit} in $H^1(\Sigma_1)$. Finally, we can show that the previous solution to \eqref{eqlimit}  is the solution to \eqref{eqlimweak} also in the case in which $\Omega$ is only $C^{1,1}$. By a change of variables and coarea formula (see \autoref{itformulas}) we can rewrite the integral on $\Sigma_1$ as 
\[\int_{\Sigma_1} g(z)J_0(z)\,dz=\int_{\partial\Omega}\int_0^{h(\sigma)} g(\sigma+t\nu_0)\,dt\,d\Hn(\sigma)\]
and the integral on $\partial\Omega_1$ as
\[\int_{\partial\Omega_1} g(\xi) \dfrac{J_0(\xi)}{\sqrt{1+\abs{\nabla h}^2}}\,d\Hn(\xi)=\int_{\partial\Omega} g(\sigma+h(\sigma)\nu_0)\,d\Hn(\sigma).\]
Then by direct computation, we have that

\[
\nabla \tilde{u}\cdot\nu_0=-\dfrac{\beta w(\sigma(z))}{1+\beta h(z)}
\]
so that, for every smooth function $\varphi\in H^1(\Sigma_1)$ with $\varphi=0$ on $\partial\Omega$,
\[\begin{split}
    \int_{\Sigma_1} (\nabla\tilde{u}_0\cdot\nu_0)(\nabla\varphi\cdot\nu_0)J_0(z)\,dz&=-\int_{\partial\Omega}\dfrac{\beta w(\sigma}{1+\beta h}\int_0^{h(\sigma)} \dfrac{d}{dt}(\varphi(\sigma+t\nu_0))\,dt\,d\Hn(\sigma)\\[15 pt]
&=-\int_{\partial\Omega}\dfrac{\beta w(\sigma}{1+\beta h}\int_0^{h(\sigma)}\varphi(\sigma+h(\sigma)\nu_0)\,d\Hn\\[15 pt]
&=-\beta\int_{\partial\Omega}\tilde{u}_0(\sigma+h(\sigma)\nu_0)\varphi(\sigma+h(\sigma)\nu_0)\,d\Hn\\[15 pt]
    &=-\beta\int_{\partial\Omega_1} \tilde{u}_0(\zeta)\varphi(\zeta)\dfrac{J_0(\zeta)}{\sqrt{1+\abs{\nabla h(\zeta)}^2}}\,d\Hn(\zeta)
\end{split}\]
and $\tilde{u}$ is in fact a solution to \eqref{eqlimweak}.
\end{oss}
\subsection{Proof of \autoref{teor: tildeu}}
Let $f\in L^2(\Omega)$ be a non-negative function and consider the functional
\[
\mathcal{F}_\eps(v,h)=
\begin{cases}
    \displaystyle \int_{\Omega}\abs{\nabla v}^2\,dx+\eps\int_{\Sigma_\eps} \abs{\nabla v}^2\,dx+\beta\int_{\partial\Omega_\eps} v^2\,d\Hn-2\int_{\Omega}fv\,dx &\text{if }v\in H^1(\Omega_\eps),\\[10 pt]
    +\infty &\text{if } v\in L^2(\R^n)\setminus H^1(\Omega_\eps),
\end{cases}
\]
and let
\[
\mathcal{F}_0(v,h)=\begin{cases}\displaystyle \int_{\Omega}\abs{\nabla v}^2\,dx+\beta\int_{\partial\Omega} \dfrac{v^2}{1+\beta h}\,d\Hn-2\int_{\Omega}fv\,dx &\text{if }v\in H^1(\Omega), \\[10 pt]
+\infty &\text{if }v\in L^2(\R^n)\setminus H^1(\Omega),
\end{cases}\]
in \cite{depiniscatro} the authors prove the following
\begin{teor}
Let $\Omega\subset\R^n$ be a bounded, open set with $C^{1,1}$ boundary, and  fix a positive Lipschitz function $h\colon\partial\Omega\to\R$. Then $\mathcal{F}_\eps(\cdot,h)$ $\Gamma$-converges in the strong $L^2(\R^n)$ topology, as $\eps\to0^+$, to $\mathcal{F}_0(\cdot,h)$.
\end{teor}
For every $\eps>0$, let $u_{\eps,h}=u_\eps\in H^1(\Omega_\eps)$ be the minimizer of $\mathcal{F}_\eps(\cdot,h)$, then $u_\eps$ is a solution to the following boundary value problem
\[\begin{cases}-\Delta u_\eps= f & \text{in }\Omega,\\[3 pt]
u_\eps^-=u_\eps^+  & \text{on }\partial\Omega, \\[3 pt]
\dfrac{\partial u_\eps^-}{\partial \nu_0\hphantom{\scriptstyle{-}}}=\eps\dfrac{\partial u_\eps^+}{\partial \nu_0\hphantom{\scriptstyle{+}}} & \text{on }\partial\Omega, \\[6 pt]
\Delta u_\eps=0 & \text{in } \Sigma_\eps, \\[3 pt]
\eps\dfrac{\partial u_\eps}{\partial \nu_\eps} +\beta u_\eps=0 & \text{on } \partial \Omega_\eps,
\end{cases}\]
where $u_\eps^-$ and $u_\eps^+$ denote the trace on $u_\eps$ from the inside $\Omega$ and from the outside of $\Omega$ respectively. By the properties of $\Gamma$-convergence, we have that 
\begin{equation}
\label{eq: weakconvergence}
u_\eps \xrightharpoonup{H^1(\Omega)}u_0,
\end{equation}
where $u_0=u_{0,h}$ is the minimizer of $\mathcal{F}_0(\cdot,h)$. Namely, $u_0$ is the solution to
\[\begin{cases}
    -\Delta u_0=f &\text{in }\Omega,\\[5 pt]
    \dfrac{\partial u_0}{\partial \nu_0}+\dfrac{\beta}{1+\beta h} u_0=0 &\text{on }\partial\Omega.
\end{cases}\]

We now prove the weak convergence of the family of functions $\tilde{u}_\eps$ defined in \autoref{teorema1}. The proof revolves around some energy estimates analogous to the one proved in \cite{BCF} for the solutions to a similar boundary value problem with a transmission condition. 
\begin{teor}
\label{teor: C11energy}
Let $\Omega\subset\R^n$ be a bounded, open set with $C^{1,1}$ boundary, and  fix a positive function $h\in C^{1,1}(\Gamma_{d_0})$ and $\nabla h\cdot \nu_0=0$. Then there exists positive constants $\eps_0(\Omega)$, and $C(\Omega,h,\beta,f)$ such that if
\begin{equation}
    \eps\norma{h}_{C^{0,1}}\le \eps_0
\end{equation}
and $u_\eps$ is the weak solution to \eqref{problema}, then
\begin{equation}
        \label{eq: H2estforuappr}
        \int_\Omega \abs{D^2 u_\eps}^2\,dx+\eps\int_{\Sigma_\eps} \abs{D^2 u_\eps}^2\,dx+\beta\int_{\partial\Omega_\eps} \abs{\nabla^{\partial\Omega_\eps} u_\eps}^2\,d\Hn\le C.
    \end{equation}
\end{teor}
\begin{oss}
    We want to point out that the assumption
    \[
        \nabla h(x)\cdot \nu_0(x) = 0 \qquad \forall x\in\Gamma_{d_0}
    \]
    is not necessary to prove \autoref{teor: C11energy}, but it makes the computations easier.
    \end{oss}   
For simplicity's sake, we postpone the technicalities of the proof of \autoref{teor: C11energy} to \autoref{EE}, and we directly prove \autoref{teor: tildeu}. Our aim is to have uniform $H^1$ estimates for $\tilde{u}_\eps$, so we first show an immediate consequence of \autoref{teor: C11energy}. For every $x\in\Gamma_{d_0}$, we denote by
\[
\nabla^{\partial\Omega}u_\eps(x)=\nabla {u}_\eps(x)-\left(\nabla {u}_\eps(x)\cdot\nu_0(x)\right)\nu_0(x),
\]
and we have the following.
\begin{cor}
\label{teor:esttan}
    There exists a positive constant $C=C(\Omega, h, \beta, f)$ such that 
    \[
    \int_{\Sigma_\eps} \abs{\nabla ^{\partial\Omega} u_\eps}^2\,dx\le \eps C.
    \]
\end{cor}
\begin{proof}
    We start by proving that on $\partial\Omega_\eps$
\begin{equation}
\label{eq: tanggrad}
   \abs{\nabla^{\partial\Omega_\eps} u_\eps}^2\ge (1-\eps^2)\abs{\nabla^{\partial\Omega} u_\eps}^2.
\end{equation}
Indeed, for every $\xi\in\partial\Omega_\eps$, if $\abs{\nabla h(\xi)}=0$, then $\nu_0=\nu_\eps$ and the inequality is trivially true; on the other hand, if $\abs{\nabla h(\xi)}\ne0$ we have that  
\[
    \abs{\nabla^{\partial\Omega_\eps} u_\eps}^2=\abs{\nabla u_\eps}^2-\abs{\nabla u_\eps\cdot \nu_\eps}^2,
\]
and for every $\eta>0$
\[\abs{\nabla u_\eps\cdot \nu_\eps}^2\le \dfrac{(1+\eta)\abs{\nabla u_\eps\cdot\nu_0}^2+\left(1+\dfrac{1}{\eta}\right)\eps^2 \abs{\nabla u_\eps\cdot \nabla h}^2}{1+\abs{\nabla h}^2}.\]
Moreover $\nabla h\cdot\nu_0=0$, so that
\[
\abs{\nabla u_\eps\cdot \nabla h}^2=\abs{\nabla^{\partial\Omega} u_\eps\cdot \nabla h}^2\le(\abs{\nabla u_\eps}^2-\abs{\nabla u_\eps\cdot\nu_0}^2)\abs{\nabla h}^2.
\]
Therefore, for every $\eta>0$
\[\abs{\nabla^{\partial\Omega_\eps} u_\eps}^2\ge\left(1-\dfrac{\left(1+\dfrac{1}{\eta}\right)\eps^2\abs{\nabla h}^2}{1+\abs{\nabla h}^2}\right)\abs{\nabla u_\eps}^2-\dfrac{1+\eta-\left(1+\dfrac{1}{\eta}\right)\eps^2\abs{\nabla h}^2}{1+\abs{\nabla h}^2}\abs{\nabla u_\eps\cdot\nu_0}^2,\]
finally, letting $\eta=\abs{\nabla h}^2$, we have 
\[\abs{\nabla^{\partial\Omega_\eps} u_\eps}^2\ge(1-\eps^2)(\abs{\nabla u_\eps}^2-\abs{\nabla u_\eps\cdot\nu_0}^2)=(1-\eps^2)\abs{\nabla^{\partial\Omega} u_\eps}^2.\]

Then, by \autoref{teor: C11energy} and \eqref{eq: tanggrad}, we have that
\begin{equation}
    \label{eq: energest}
    \eps\int_{\Sigma_\eps}\abs{D^2 u_\eps}^2\,dx+\beta\int_{\partial\Omega_\eps} \abs{\nabla^{\partial\Omega} u_\eps}^2\,d\Hn\le C.
\end{equation}
For $x\in\Sigma_\eps$ denote by 
\[
\xi(x)=\sigma(x)+\eps h(x)\nu_0(x)\in\partial\Omega_\eps,
\]
so that for $\Ln$-a.e. $x\in\Sigma_\eps$, we have that
\[\nabla^{\partial\Omega} u_\eps(x)=\nabla^{\partial\Omega} u_\eps(\xi(x))-\int_{d(x)}^{\eps h(x)}\dfrac{d}{dt}\left(\nabla^{\partial\Omega} u_\eps(\sigma(x)+t\nu_0(x))\right)\,dt.\]
Hence,
\[\abs{\nabla^{\partial\Omega} u_\eps}(x)^2\le C(h)\left(\abs{\nabla^{\partial\Omega} u_\eps}^2(\xi(x))+\eps\int_0^{\eps h(x)}\abs{D^2 u_\eps}^2(\sigma(x)+t\nu_0(x))\,dt\right),\]
and integrating over $\Sigma_\eps$, using \eqref{eqivint1} and \eqref{eqivint2}, and \eqref{eq: energest}, we get
\[\begin{split}\int_{\Sigma_\eps}\abs{\nabla^{\partial\Omega} u_\eps}^2(x)\,dx\le& \,C\int_{\partial\Omega} \int_0^{\eps h(\sigma)} \abs{\nabla^{\partial\Omega} u_\eps}^2(\sigma+\eps h(\sigma)\nu_0)\,ds\,d\Hn(\sigma)+\\[15 pt]
&+C\eps\int_{\partial\Omega}\int_0^{\eps h(\sigma)}\int_0^{\eps h(\sigma)}\abs{D^2 u_\eps}^2(\sigma(x)+t\nu_0(x))\,dt\,ds\,d\Hn(\sigma)\\[15 pt]
\le&\,\eps C\left(\int_{\partial\Omega_\eps} \abs{\nabla^{\partial\Omega} u_\eps}^2\,d\Hn+\eps\int_{\Sigma_\eps}\abs{D^2 u_\eps}^2\,dx\right)\\[15 pt]
\le&\,\eps C.\end{split}\]
\end{proof}

We are now in a position to prove the $H^1$ convergence of the family $\tilde{u}_\eps$.
\begin{proof}[Proof of \autoref{teor: tildeu}]
We recall that we are assuming without loss of generality $\norma{h}_\infty<d_0$, so that $d\in C^{1,1}(\Sigma_1)$. To prove the equiboundedness in $H^1(\Omega_1)$, we decompose $\nabla \tilde{u}_\eps$ into its normal part \[
(\nabla \tilde{u}_\eps\cdot \nu_0 )\,\nu_0,
\]
and its tangential part
\[
\nabla^{\partial\Omega} \tilde{u}_\eps:=\nabla \tilde{u}_\eps-\left(\nabla \tilde{u}_\eps\cdot\nu_0\right)\nu_0.
\]    
Using \autoref{lem: areaformulasstretch}, since $J_\eps$ and $J_\eps^\tau$ are equibounded, we can find a positive constant $C=C(\Omega,h)$ such that
\begin{equation}
\label{eq: energy1}
\int_{\Sigma_1}\abs{\nabla^{\partial\Omega}\tilde{u}_\eps}^2\,dz\le \dfrac{C}{\eps}\int_{\Sigma_\eps}\abs{\nabla^{\partial\Omega} u_\eps}^2\,dx,
\end{equation}
\begin{equation}
\label{eq: energy2}
\int_{\Sigma_1}\abs{\nabla \tilde{u}_\eps\cdot\nu_0}^2\,dz\le\eps C\int_{\Sigma_\eps} \abs{\nabla u_\eps}^2\,dx,    
\end{equation}
and
\begin{equation}
\label{eq: energy3}
\int_{\partial\Omega_1}\tilde{u}_\eps^2\,d\Hn\le C\int_{\partial\Omega_\eps}u_\eps^2\,d\Hn.
\end{equation}
Moreover, by the weak convergence of $u_\eps$ in $H^1(\Omega)$, $u_\eps$ are equibounded in $L^2(\Omega)$, while by the minimality
\begin{equation}\label{boundtraceps}\begin{split}
\int_\Omega \abs{\nabla u_\eps}^2\,dx+\eps\int_{\Sigma_\eps} \abs{\nabla u_\eps}^2\,dx+\beta\int_{\partial\Omega_\eps} u_\eps^2\,d\Hn &\le \mathcal{F}_\eps(0)+2\int_\Omega f u_\eps\,dx\\[7 pt] &
\le \int_\Omega (f^2+u_\eps^2)\,dx\\[7 pt] &\le C
\end{split}\end{equation}
Joining the inequalities \eqref{eq: energy1}, \eqref{eq: energy2}, \eqref{eq: energy3}, and \eqref{boundtraceps} with the energy estimates in \autoref{teor:esttan}, we have that for some positive constant $C(\Omega,h,\beta)$
\[\int_{\Omega_1}\abs{\nabla \tilde{u}_\eps}^2\,dz+\beta\int_{\partial\Omega_1}\tilde{u}_\eps^2\,d\Hn\le C.\]
By Poincaré's inequality with boundary term, we have that $\tilde{u}_\eps$ are equibounded in $H^1(\Omega_1)$. Therefore, up to a subsequence, $\tilde{u}_\eps$ converges weakly in $H^1(\Omega_1)$ to some function $\tilde u_{0}$. In particular, by the weak convergence \eqref{eq: weakconvergence} inside $\Omega$, we have $\tilde{u}_0=u_0$ in $\Omega$. On the other hand, in $\Sigma_1$, using \autoref{prop: weaklim}, we get that $\tilde{u}_0$ is a solution to \eqref{eqlimweak}, so that \autoref{oss: explsol} ensures that
\[
\tilde{u}_0(z)=\tilde{u}_0(\sigma(z))\left(1-\dfrac{\beta d(z)}{1+\beta h(z)}\right).
\]
Finally, since $\tilde{u}_0\in H^1(\Omega_1)$, then it cannot jump across $\partial\Omega$, so that, since $\tilde{u}_0=u_0$ in $\Omega$, we necessarily have that $\tilde{u}_0(\sigma(z))=u_0(\sigma(z))$ for a.e. $z\in\Sigma_1$, and the theorem is proved.
\end{proof}

\section{Asymptotic Development}\label{Gamma}

In this section, we study the first-order asymptotic development by $\Gamma$-convergence of the functional $\mathcal{F}_\eps(\cdot,h)$. Let us recall the notation introduced in \autoref{sec: distcurv}
\[
\gamma_t=\partial(\Omega\cup\Gamma_t)=\Set{x\in\R^n|d(x)=t}\setminus\Omega.
\]
\begin{oss}\label{itformulas}
Using the coarea formula (\autoref{coarea}) with the distance function $d$, we have that for every $g\in L^1(\Omega_\eps)$ 
\[ \int_{\Sigma_\eps}g(x)\,dx=\int_0^{+\infty}\int_{\gamma_t} g(\xi)\,\chi_{\Sigma_\eps}\!(\xi)\,d\Hn(\xi)\,dt.\]
Let 
\[
\phi_t\colon x\in\Gamma_{d_0}\mapsto x+t\nu_0(x)\in\Gamma_{td_0},
\]
then $\gamma_t=\phi_t(\partial\Omega)$, and by the area formula on surfaces (\autoref{teor:area})
\[\begin{split}\int_{\Sigma_\eps}g(x)\,dx&=\int_0^{+\infty}\int_{\partial\Omega} g(\sigma+t\nu_0)\chi_{\Sigma_\eps}(\sigma+t\nu_0) J^\tau\phi_t(\sigma)\,d\Hn(\sigma)\,dt\\[10 pt]
&=\int_{\partial\Omega} \int_0^{\eps h(\sigma)}g(\sigma+t\nu_0)\prod_{i=1}^{n-1}(1+tk_i(\sigma))\,dt\,d\Hn(\sigma).
\end{split}\]
Similarly,
\[
\int_{\partial\Omega_\eps}g(\xi)\,d\Hn(\xi)=\int_{\partial\Omega} g(\sigma+\eps h\nu_0)\prod_{i=1}^{n-1}(1+\eps h(\sigma) k_i(\sigma))\sqrt{1+\eps^2\abs{\nabla h}^2}\,d\Hn(\sigma).
\]
so that we have
\begin{equation}\label{eq:intsigma} \int_{\Sigma_\eps}g(x)\,dx=\int_{\partial\Omega}\int_0^{\eps h(\sigma)} g(\sigma+t\nu_0)\left(1+tH(\sigma)+\eps^2 R_1(\sigma,t,\eps)\right)\,dt\,d\Hn\end{equation}
and 
\begin{equation}\label{eq:intdesigma} \int_{\partial\Omega_\eps}g(\sigma)\,d\Hn=\int_{\partial\Omega} g(\sigma+\eps h\nu_0)\left(1+\eps h(\sigma)H(\sigma)+\eps^2 R_2(\sigma,\eps)\right)\,d\Hn,\end{equation}
where the remainder terms $R_1$ and $R_2$ are bounded functions. In other words, there exists $Q=Q(\Omega,h)>0$ such that $\abs{R_1},\abs{R_2}\le Q$.

In particular, notice that there exists a positive constant $C=C(\Omega,\norma{h}_{0,1})$ such that for every $0<\eps<1$
\begin{equation}\label{eqivint1}\dfrac{1}{C}\int_{\Sigma_\eps} g\,dx\le\int_{\partial\Omega}\int_0^{\eps h(\sigma)} g(\sigma+t\nu_0)\,dt\,d\Hn\le C\int_{\Sigma_\eps}g\,dx,\end{equation}
and
\begin{equation}\label{eqivint2}\dfrac{1}{C}\int_{\partial\Omega_\eps} g\,dx\le\int_{\partial\Omega} g(\sigma+\eps h\nu_0)\,d\Hn\le C\int_{\partial\Omega_\eps}g\,dx.\end{equation}
\end{oss}

We can now follow the lead of \cite{optimalthin} to prove \autoref{teorema1}.

 \begin{proof}[Proof of \autoref{teorema1}]
 We start by proving the $\Gamma$-liminf inequality: 
without loss of generality, we can prove the inequality for the sequence of minimizers $u_\eps$. Here we recall the definitions of $\mathcal{F}_\eps$ and $\mathcal{F}_0$, omitting the dependence on $h$.
 \begin{equation}
 \label{eq: defFeps}
 \mathcal{F}_\eps(u_\eps)=\int_\Omega \abs{\nabla u_\eps}^2\,dx+\eps\int_{\Sigma_\eps} \abs{\nabla u_\eps}^2\,dx+\beta\int_{\partial\Omega_\eps} u_\eps^2\,d\Hn-2\int_\Omega fu_\eps\,dx,
 \end{equation}
 \begin{equation}
  \label{eq: defF0}
 \mathcal{F}_0(u_0)=\int_\Omega \abs{\nabla u_0}^2\,dx+\beta\int_{\partial\Omega} \dfrac{u_0^2}{1+\beta h}\,d\Hn-2\int_\Omega fu_0\,dx.
 \end{equation} 
 Moreover, notice that, by minimality of $u_0$,
 \begin{equation}
     \label{eq: Fe-F0}
     \mathcal{F}_\eps(u_\eps)-\mathcal{F}_0(u_0)\ge \eps\int_{\Sigma_\eps} \abs{\nabla u_\eps}^2\,dx+\beta\int_{\partial\Omega_\eps} u_\eps^2\,d\Hn-\beta\int_{\partial\Omega} \frac{u_\eps^2}{1+\beta h}\,d\Hn
 \end{equation}
 By \eqref{eq:intsigma} and \eqref{eq:intdesigma} we have \begin{equation}\label{cov1}\int_{\Sigma_\eps}\abs{\nabla u_\eps}^2\,dx\ge\int_{\partial\Omega}\int_0^{\eps h(\sigma)} \abs{\nabla u_\eps(\sigma+t\nu_0)}^2\left(1+tH(\sigma)-\eps^2 Q\right)\,dt\,d\Hn\end{equation}
and
\begin{equation}\label{cov2}\dfrac{\beta}{\eps}\int_{\partial\Omega_\eps} u_\eps^2\,\Hn\ge\dfrac{\beta}{\eps}\int_{\partial\Omega} u_\eps^2(\sigma+\eps h(\sigma)\nu_0(\sigma))\left(1+\eps h(\sigma)H(\sigma)-\eps^2 Q\right)\,d\Hn.\end{equation}
We choose $\eps$ sufficiently small, so that for every $\sigma\in\partial\Omega$, and $0<t<\eps h(\sigma)$, we have that $1+tH(\sigma)>0$, and, in particular, using Holder's inequality and integrating by parts,
\[\begin{split}\int_0^{\eps h} \abs{\nabla u_\eps(\sigma+t\nu_0)}^2\left(1+tH\right)\,dt&\ge\dfrac{1}{\eps h}\!\left(\int_0^{\eps h} \abs{\nabla u_\eps(\sigma+t\nu_0)}\sqrt{1+tH}\,dt\right)^2\\[10pt]
&\ge\dfrac{1}{\eps h}\!\left(\int_0^{\eps h}\dfrac{d}{dt}( u_\eps(\sigma+t\nu_0))\sqrt{1+tH}\,dt\right)^2\\[10 pt]
&\ge\frac{1}{\eps h}\!\left(u_\eps(\sigma+\eps h\nu_0)\sqrt{1+\eps h H}-\!\left(u_\eps(\sigma)+\!\!\displaystyle\int_0^{\eps h}\dfrac{H u_\eps(\sigma+t\nu_0)}{2\sqrt{1+tH}}\,dt\right)\right)^{\!2}\!.
\end{split}\]
We then apply Young's inequality to the absolute value of the double product, having that for every $\lambda>0$
\begin{equation}
\label{eq: lowergrad}
\begin{split}\int_0^{\eps h(\sigma)} \abs{\nabla u_\eps(\sigma+t\nu_0)}^2\left(1+tH(\sigma)\right)\,dt\ge& \dfrac{(1-\lambda)(1+\eps h H)u_\eps(\sigma+\eps h\nu_0)^2}{\eps h}\\[10 pt]&+\dfrac{1}{\eps h}\left(1-\dfrac{1}{\lambda}\right)\left(u_\eps(\sigma)+\int_0^{\eps h(\sigma)}\dfrac{H u_\eps(\sigma+t\nu_0)}{2\sqrt{1+tH}}\,dt\right)^2. \end{split}
\end{equation}
We then have, joining \eqref{eq: Fe-F0}, \eqref{cov1}, \eqref{eq: lowergrad}, and \eqref{cov2},
\begin{equation}
    \label{4asterischi}
    \begin{split}
        \delta\mathcal{F}_\eps(u_\eps)=&\dfrac{\mathcal{F}_\eps(u_\eps)-\mathcal{F}_0(u_0)}{\eps}\\[7 pt] 
        \ge&\int_{\partial\Omega}\dfrac{1}{\eps h(\sigma)}(1-\lambda+\beta h)(1+\eps hH)u_\eps^2(\sigma+\eps h\nu_0)\,d\Hn\\[10 pt]%
        &+\int_{\partial\Omega}\dfrac{1}{\eps h}\left(\left(1-\dfrac{1}{\lambda}\right)\left(u_\eps(\sigma)+\int_0^{\eps h}\dfrac{H u_\eps(\sigma+t\nu_0)}{2\sqrt{1+tH}}\,dt\right)^2-\dfrac{\beta h \, u_\eps^2(\sigma)}{1+\beta h}\right)\,d\Hn-Q\eps R(\eps,u_\eps)
    \end{split}
\end{equation} where, if $\eps$ is small enough, 
\[\begin{split}R(\eps,u_\eps)&=\eps\int_{\partial\Omega}\int_0^{\eps h(\sigma)} \abs{\nabla u_\eps(\sigma+t\nu_0)}^2\,d\Hn+\beta\int_{\partial\Omega} u_\eps(\sigma+\eps h(\sigma)\nu_0(\sigma))^2\,d\Hn\\[10 pt]
&\le C\int_\Omega fu_\eps\,dx<C.
\end{split}\]
Letting $\lambda=\lambda(\sigma)=1+\beta h(\sigma)$ in \eqref{4asterischi}, and using the inequality $(a+b)^2\ge a^2+2ab$, joint with the fact that $1-\lambda^{-1}>0$,
\[
    \delta\mathcal{F}_\eps(u_\eps)\ge\int_{\partial\Omega}\dfrac{\beta H u_\eps(\sigma)}{\eps (1+\beta h)}\int_0^{\eps h(\sigma)}\dfrac{u_\eps(\sigma+t\nu_0)}{\sqrt{1+tH}}\,dt\,d\Hn+O(\eps).
\]
Moreover, for every $t\in(0,\eps \norma{h}_\infty)$ we have that $(1+tH)^{-1/2}=1+O(\eps)$, so that
\begin{equation}
\label{stimaliminf}
    \delta\mathcal{F}_\eps(u_\eps)\ge\beta \int_{\partial\Omega}\dfrac{ H u_\eps(\sigma)}{(1+\beta h)}\int_0^{h(\sigma)}\tilde{u}_\eps(\sigma+t\nu_0)\,dt\,d\Hn+O(\eps).
\end{equation}
Finally, by \autoref{teor: tildeu} we get
\[
\tilde{u}_\eps \xrightarrow[]{L^2(\Sigma_1)}\tilde{u}_0,
\]
so that
\[
\int_0^{h(\sigma)}\tilde{u}_\eps(\sigma+t\nu_0)\,dt\xrightarrow[]{L^2(\partial\Omega)}\int_0^{h(\sigma)}\tilde{u}_0(\sigma+t\nu_0)\,dt.
\]
Indeed, by \eqref{eqivint1},
\[\int_{\partial\Omega}\left(\int_0^{h(\sigma)}\left(\tilde{u}_\eps(\sigma+t\nu_0)-\tilde{u}_0(\sigma+t\nu_0)\right)\,dt\right)^2\,d\Hn\le C\int_{\Sigma_1}(\tilde{u}_\eps-\tilde{u}_0)^2\,dz. \]
Therefore, passing to the limit in \eqref{stimaliminf}, and using the explicit expression of $\tilde{u}_0$, we get
\[
\liminf_{\eps\to0^+}\delta\mathcal{F}_\eps(u_\eps)\ge\beta\int_{\partial\Omega}\dfrac{hH(2+\beta h)}{2(1+\beta h)^2}u_0^2(\sigma)\,d\Hn,
\]
and the $\Gamma$-liminf is proved.
\medskip

We now prove the $\Gamma$-limsup inequality.

Let
\[
\varphi_\eps(x)=\begin{dcases}u_0(x) &\text{if }x\in\Omega, \\[5 pt]
u_0(\sigma(x))\left(1-\dfrac{\beta d(x)}{\eps(1+\beta h(x))}\right) &\text{if }x\in\Sigma_\eps,\\[5 pt]
0 &\text{if }x\in\R^n\setminus\Omega_\eps,
\end{dcases}\]
where we recall that  $\nabla h\cdot\nu_0=0$. We have that $\varphi_\eps\in H^1(\Omega)$ and $\varphi_\eps$ converges in $L^2(\R^n)$, to $u_0\chi_\Omega$.
Since $\varphi_\eps\equiv u_0$ in $\Omega$, then by definition of the functionals \eqref{eq: defFeps}, and \eqref{eq: defF0}, we can write
\begin{equation}
    \label{eq: Fe-F0limsup}
        \mathcal{F}_\eps(\varphi_\eps)-\mathcal{F}_0(u_0)= \eps\int_{\Sigma_\eps} \abs{\nabla \varphi_\eps}^2\,dx+\beta\int_{\partial\Omega_\eps} \varphi_\eps^2\,d\Hn-\beta\int_{\partial\Omega} \frac{u_0^2}{1+\beta h}\,d\Hn.
\end{equation}
 Computing the gradient of $\varphi_\eps$, for any $x\in\Sigma_\eps$,
\[
\abs{\nabla \varphi_\eps}^2(x)\le\dfrac{\beta^2 u_0^2(\sigma(x))}{\eps^2(1+\beta h)^2}+C\left(\abs{\nabla u_0(\sigma(x))}^2+u_0^2(\sigma(x))\right),
\]
where $C=C(h,\beta)$. 
 Hence, by \eqref{eq:intsigma}, and noticing that $\sigma+t\nu_0\in \Gamma_{d_0}$ implies $d(\sigma+t\nu_0(\sigma))=t$, we get
\begin{equation}\label{stimagrad}
\begin{split}
    \eps\int_{\Sigma_\eps} \abs{\nabla \varphi_\eps}^2\,dx\le&\dfrac{\beta^2}{\eps}\int_{\Sigma_\eps}\dfrac{u_0^2(\sigma)}{(1+\beta h)^2}\,dx+\eps C\int_{\Sigma_\eps} \left(\abs{\nabla u_0(\sigma(x))}^2+u_0^2(\sigma(x))\right)\,dx \\[10 pt]
    \le&\beta^2\int_{\partial\Omega}\dfrac{u_0^2(\sigma)h}{(1+\beta h)^2}\left(1+\dfrac{\eps h H}{2}\right)\,d\Hn+O(\eps^2).
\end{split}
\end{equation}
On the other hand, for every $\sigma\in\partial\Omega$,
\[
\varphi_\eps(\sigma+\eps h(\sigma)\nu_0(\sigma))=\frac{u_0(\sigma)}{1+\beta h(\sigma)},
\]
from which we get
\begin{equation}\label{phi2}\begin{split}
\beta\int_{\partial{\Omega_\eps}} \varphi_\eps^2\,d\Hn\le\beta \int_{\partial\Omega}\dfrac{u_0^2(\sigma)}{(1+\beta h)^2}(1+\eps h H)\,d\Hn
+O(\eps^2).
\end{split}
\end{equation}
Finally, joining \eqref{eq: Fe-F0limsup}, \eqref{stimagrad}, \eqref{phi2}, we have
\[
\begin{split}
\delta\mathcal{F}_\eps(\varphi_\eps)=\dfrac{\mathcal{F}_\eps(u_\eps)-\mathcal{F}_0}{\eps}&\le\beta\int_{\partial\Omega}\dfrac{u_0^2(\sigma) h H}{(1+\beta h)^2}\left(\dfrac{\beta h}{2}+1\right)\,d\Hn +O(\eps)\\[10 pt]
&=\beta\int_{\partial\Omega}\dfrac{hH(2+\beta h)}{2(1+\beta h)^2}\,u_0^2(\sigma)\,d\Hn +O(\eps)
\end{split}
\]
so that
\[\limsup_{\eps\to0^+}\,\delta\mathcal{F}_\eps(\varphi_\eps)\le\beta\int_{\partial\Omega}\dfrac{h H(2+\beta h)}{2(1+\beta h)^2}\,u_0^2(\sigma)\,d\Hn\]
and the $\Gamma$-limsup inequality is proved.
 \end{proof}

\section{Energy estimates}\label{EE}
The aim of this section is to prove \autoref{teor: C11energy}. The proof is mainly divided in two steps:
\begin{enumerate}[label= \textbf{Step \arabic*.}]
    \item We prove some $H^1$ uniform estimates for the minimizer $u_\eps$ (see \autoref{cor: firstorderestim}).
    \item We prove the uniform $H^2$ estimates in \autoref{teor: C11energy} for the minimizers $u_\eps$ using a local argument similar to the approach in \cite{BCF}: we focus on small neighborhoods $V_{\sigma_0}$ of points $\sigma_0\in\partial\Omega$, and we construct a diffeomorphism $\Phi_{\sigma_0}$ that flattens both $\partial\Omega$ and $\partial\Omega_\eps$; on the flattened set $\Phi(V_{\sigma_0}\cap\Sigma_\eps)$ we are able to compute energy estimates of the new functions $v_\eps=u_\eps\circ \Phi^{-1}$.
\end{enumerate}
As in the previous sections, let $u_\eps$ be the minimizer to
\begin{equation}
 \label{eq: min}
 \min\Set{\mathcal{F}_\eps(v) | v\in H^1(\Omega_\eps)},
\end{equation}
where
\[
\mathcal{F}_\eps(v)=\int_\Omega \abs{\nabla v}^2\,dx-2\int_\Omega f v\,dx+\eps\int_{\Sigma_\eps} \abs{\nabla v}^2\,dx+\beta\int_{\partial\Omega_\eps} v^2\,d\Hn,
\]
and $f\in L^2(\Omega)$. The functions $u_\eps$ can be equivalently seen as the unique solutions to the equations
\begin{equation}
\label{eq: weakeq}
\int_\Omega \nabla u_\eps\cdot \nabla \varphi\, dx+\eps\int_{\Sigma_\eps}\nabla u_\eps \cdot \nabla \varphi \, dx+\beta \int_{\partial\Omega_\eps}u_\eps\varphi\,d\Hn = \int_\Omega f\varphi\,dx, 
\end{equation}
for every $\varphi\in H^1(\Omega_\eps)$. We notice that for every $\eps>0$, by standard elliptic regularity, $u_\eps\in H^2_{\loc}(\Omega_\eps\setminus\partial\Omega)$.

\subsection{Step 1: \texorpdfstring{$H^1$}{} uniform estimates}
We now estimate the $H^1$ norm of $u_\eps$ in terms of $\eps$.
\begin{lemma}
Let $\Omega\subset\R^n$ be a bounded, open set with $C^{1,1}$ boundary, and  fix a positive function $h\in C^{0,1}(\Gamma_{d_0})$ such that $\nabla h\cdot\nu_0=0$. Then there exist positive constants $\eps_0(\Omega)$, and $C(\Omega,\norma{h}_{C^{0,1}},\beta,f)$ such that if
\begin{equation}
    \eps\norma{h}_{\infty}\le \eps_0,
\end{equation}
then for every $v\in H^1(\Omega_\eps)$
\begin{equation}\label{eq:epspoinc}\int_{\Omega_\eps} v^2\,dx\le C\left[\int_\Omega \abs{\nabla v}^2\,dx+\eps\int_{\Sigma_\eps} \abs{\nabla v}^2\,dx+\beta\int_{\partial\Omega_\eps} v^2\,d\Hn\right].\end{equation}
\end{lemma}

\begin{proof}

Up to using a density argument, it is enough to prove the assertion for every smooth function in $H^1(\Omega_\eps)$. Let $v\in C^1(\overline{\Omega_\eps})$. For every $x\in\Sigma_\eps$ we recall that we can represent%
\[
x=\sigma(x)+d(x)\nu_0(x), %
\]
and we define
\[
\xi(x):=\sigma(x)+\eps h(x)\, \nu_0(x).
\]
This construction allows us to write for every $x\in\Sigma_\eps$
\[
v(x)=v(\xi(x))-\int_{d(x)}^{\eps h(x)} \dfrac{\partial}{\partial\nu_0} v(\sigma(x)+t\nu_0(x))\,dt.\]
Hence, by Young and Hölder inequalities,
\begin{equation}
\label{eq: v2estim}
v^2(x)\le2v^2(\xi(x))+2\eps\norma{h}_\infty\int_0^{\eps h(x)}\abs{\nabla v}^2(\sigma(x)+t\nu_0(x))\,dt.
\end{equation}
Integrating over $\Sigma_\eps$, using \eqref{eqivint1} and \eqref{eqivint2}, and recalling that $\xi(x)=\xi(\sigma(x))$, we can find a positive constant $C=C(\Omega,\norma{h}_{C^{0,1}})$ such that
\begin{equation}\label{eq: epspoinc0}\begin{split}\int_{\Sigma_\eps} v^2\,dx\le& C\int_{\partial\Omega}\int_0^{\eps h(\sigma)} v^2(\xi(\sigma))\,ds\,d\Hn(\sigma)\\[10 pt]
&+\eps C\int_{\partial\Omega}\int_0^{\eps h(\sigma)} \int_0^{\eps h(\sigma)} \abs{\nabla v}^2(\sigma+t\nu_0)\,dt\,ds\,d\Hn(\sigma)\\[10 pt]
\le&\eps C\left(\int_{\partial\Omega_\eps}v^2\,d\Hn+\eps\int_{\Sigma_\eps}\abs{\nabla v}^2\,dx\right).
\end{split}\end{equation}
Similarly, we can integrate \eqref{eq: v2estim} over $\partial\Omega$ and have
\begin{equation}\label{eq: partiaomega}\int_{\partial\Omega} v^2\,d\Hn\le C\left(\int_{\partial\Omega_\eps}v^2\,d\Hn+\eps\int_{\Sigma_\eps}\abs{\nabla v}^2\,dx\right).\end{equation}
From the Poincaré inequality with trace term in $\Omega$ and the Bossel-Daners inequality, we have 
\begin{equation}
\label{eq: epspoinc1}
\int_\Omega v^2\,dx\le C_p(\abs{\Omega})\left(\int_\Omega \abs{\nabla v}^2\,dx+\int_{\partial\Omega} v^2\,d\Hn\right),
\end{equation}
so that joining \eqref{eq: epspoinc0} and \eqref{eq: epspoinc1}, and using \eqref{eq: partiaomega} we have the assertion.

\end{proof}

\begin{cor}
\label{cor: firstorderestim}
Let $\Omega\subset\R^n$ be a bounded, open set with $C^{1,1}$ boundary, and  fix a positive function $h\in C^{0,1}(\Gamma_{d_0})$ such that $\nabla h\cdot\nu_0=0$. Then there exists a positive constant $C=C(\Omega, h,\beta,f)$
such that if $u_\eps\in H^1(\Omega_\eps)$ is the minimizer to \eqref{eq: min}, then
\begin{equation}
\label{eq: firstderbound}
    \int_\Omega \abs{\nabla u_\eps}^2\,dx+\eps\int_{\Sigma_\eps} \abs{\nabla u_\eps}^2\,dx+\beta\int_{\partial\Omega_\eps} u_\eps^2\,d\Hn\le C,
\end{equation}
and
\begin{equation}
\label{eq: L2bound}
\int_{\Omega_\eps} u_\eps^2\,dx\le C.
\end{equation}
\end{cor}
\begin{proof}
For every $\eta>0$, we can write
\[\begin{split}\int_\Omega \abs{\nabla u_\eps}^2\,dx+\eps\int_{\Sigma_\eps} \abs{\nabla u_\eps}^2\,dx+\beta\int_{\partial\Omega_\eps} u_\eps^2\,d\Hn&\le \mathcal{F}_\eps(0)+2\int_\Omega f u_\eps\,dx\\[15 pt]
&\le \eta\int_\Omega f^2\,dx+\dfrac{1}{\eta}\int_\Omega u_\eps^2\,dx.\end{split}\]
Using \eqref{eq:epspoinc}, for a suitable choice of $\eta$, we get the result.
\end{proof}

\subsection{Flattening the boundaries}
Here we aim to construct a flattening diffeomorphism that locally transforms $\partial\Omega$ and $\partial\Omega_\eps$ in subsets of parallel planes. To do so, we have to represent locally $\partial\Omega$ and $\partial\Omega_\eps$.
\begin{lemma}[Uniform local representation of $\partial\Omega$ and $\partial\Omega_\eps$.]
\label{lem: uniformrepr}
    Let $\Omega\subset\R^n$ be a bounded, open set with $C^{1,1}$ boundary, fix a positive function $h\in C^{0,1}(\Gamma_{d_0})$ such that $\nabla h\cdot\nu_0=0$, and let $\sigma_0\in\partial\Omega$. There exists $\eps_0=\eps_0(\Omega,\sigma_0)$ such that, if
    \[
        \eps\norma{h}_{C^{0,1}}<\eps_0,
    \]
    then there exist an open set $V$ containing $\sigma_0$ and $\sigma_0+\eps h(\sigma_0)\nu_0(\sigma_0)$, and there exist functions $g,k_\eps\colon \R^{n-1}\to\R$ such that, up to a rototranslation, 
\[
\begin{split}
\Omega\cap V&=\Set{(x',x_n)|x_n\le g(x')}\cap V,\\
\Omega_\eps\cap V&=\Set{(x',x_n)|x_n\le g(x')+ \eps k_\eps(x')}\cap V,
\end{split}
\] 
and 
\[
\begin{split}
\partial\Omega\cap V&=\Set{(x',x_n)|x_n= g(x')}\cap V,\\
\partial\Omega_\eps\cap V&=\Set{(x',x_n)|x_n=g(x')+ \eps k_\eps(x')}\cap V,
\end{split}
\] 
    \begin{proof}
    Without loss of generality we can assume that $\nu_0(\sigma_0)=\mathbf{e}_n$ and $\sigma_0=0$. We already know by definition that $\Omega$ can be represented locally near $0$, that is, there exist a neighborhood $U$ of $\sigma_0$ and a function $g\in C^{1,1}(\R^{n-1})$ with $g(0)=0$ and $\nabla g(0)=0$, such that
      \[
        \Omega\cap U=\Set{(x',x_n)\in\R^n | x_n\le g(x')}\cap U.
    \]
 For every $r\in(0,1)$ let 
    \[
        B_r'=\Set{x'\in \R^{n-1} | \abs{x'}\le r},
    \]
   and $V_r=B_r'\times [-2r,2r]$. For every $x'\in B'_r$ let
    \[
        F_{x'}: t\in(g(x'),2r]\longmapsto  d(x',t)-\eps h(x',t),
    \]
    where we recall that $d(x)$ denotes the distance from $\Omega$. Let us also recall that
    \[
     \Sigma_\eps=\Omega_\eps\setminus\overline{\Omega}=\Set{x\in\R^n | 0<d(x)<\eps h(x)}.
    \]
    The definition of $F_{x'}$ gives us the possibility to characterize the property $(x',t)\in\Omega_\eps$ in the equivalent way $F_{x'}(t)<0$. As an immediate consequence, $F_{x'}(g(x'))<0$ for every $x'\in B_r'$. 
    
    The idea of the proof can be divided in two main steps: first we show that for a right choice of $\eps_0$ the set $\partial\Omega_\eps$ cannot touch the upper base of the cylinder $V_r$ by proving $F_{x'}(2r)>0$ for every $x'\in B_r'$; next, we show that for a right choice of $\eps_0$ we have that $\partial\Omega_\eps$ can be represented as a graph over the whole $B_r'$ by showing that $F_{x'}$ is strictly increasing in $t$ for every $x'\in B_r'$ 

   Let us choose $r=r(\sigma_0,\Omega)$ small enough so that we have $V_r\subseteq U$ and, thanks to the continuity of $\nu_0$ in $\Gamma_{d_0}$
   \begin{equation}
    \label{eq: nu0en}
   \abs{\nu_0(x)-\mathbf{e}_n}<\dfrac{1}{2},
   \end{equation}
   for every $x\in V_r\cap \Gamma_{d_0}$. Note in addition that up to choosing a smaller $r$ we can assume the vertical distance between $\partial\Omega$ and the upper base of $V_r$ being greater than $r$, namely 
    \begin{equation}
    \label{eq: 2r-maxg}
        2r-\max_{x'\in B_r'}g(x')>r.
    \end{equation}
 Indeed, since $\nabla g$ is continuous and $\nabla g(0)=0$, then for small enough $r$ we also have, for every $x'\in B_r'$,
\[
\abs{\nabla g(x')}<1,
\]
 which joint with the fact that $g(0)=0$, ensures \eqref{eq: 2r-maxg}.

    Let $x'\in B'_r$, fix $t\in (g(x'),2r]$, and let $x=(x',t)$. We claim that under the assumption \eqref{eq: nu0en}, we have
    \begin{equation} 
    \label{eq: distt-g}
    d(x)> \frac{t-g(x')}{2}.
    \end{equation}
    Indeed, let $\bar{x}$ denote the point at distance $d(\bar{x})=t-g(x')$ whose projection onto $\Omega$ is $(x',g(x'))$, namely
\[
\overline{x}=(x',g(x'))+(t-g(x'))\nu_0(x',g(x')).
\]
We have
\[
\abs{x-\overline{x}}=(t-g(x'))\,\abs{\mathbf{e}_n-\nu_0(x',g(x'))}<\dfrac{t-g(x')}{2},
\]
and, using the fact that the distance $d$ is Lipschitzian with constant 1, 
\[\begin{split}d(x)&\ge d(\overline{x})-\abs{d(x)-d(\overline{x})}\\[5 pt]
&\ge d(\overline{x})-\abs{x-\overline{x}}>\dfrac{t-g(x')}{2}.
\end{split}\]
We now join \eqref{eq: distt-g}, \eqref{eq: 2r-maxg}, and $\eps\norma{h}_\infty<\eps_0$ to get that for $\eps_0<r/2$
\begin{equation}
\label{eq: Fx'(2r)}
F_{x'}(2r)>\dfrac{2r-g(x')}{2}-\eps_0>0,
\end{equation}
which concludes the first step.

    We now prove that $F_{x'}(t)$ is monotone increasing in $t$. Indeed, by the assumption \eqref{eq: nu0en}, we have that
   \[
        \nu_0(x)\cdot \mathbf{e}_n>\frac{1}{2},\] 
        and
        \[\abs{\nabla h(x)\cdot \mathbf{e}_n}=\abs{\nabla h(x)\cdot( \mathbf{e}_n-\nu_0(x))}\le\frac{\norma{\nabla h}_\infty}{2},
    \]
    so that, since $r<1$ and $\eps_0<r/2$, we have $\eps\norma{\nabla h}_{\infty}<1/2$, which ensures
    \begin{equation}
    \label{eq: Fxincreas}
        \dfrac{d}{dt}F_{x'}(t)=\nu_0(x',t)\cdot\mathbf{e}_n-\eps\nabla h(x',t)\cdot\mathbf{e}_n > \frac{1}{4}.
    \end{equation}

    Joining \eqref{eq: Fx'(2r)} and \eqref{eq: Fxincreas}, we get that for every $x'\in B_r'$ there exists a unique $t(x')\in(g(x'),2r)$ such that $(x',t)\in\Omega_\eps\cap V_r$ if and only if $t\le t(x')$, thus concluding the proof.
\end{proof}
\end{lemma}

Thanks to \autoref{lem: uniformrepr}, we can now represent both $\partial\Omega$ and $\partial\Omega_\eps$ as graphs, uniformly in $\eps$, and flatten the boundaries of $\Omega$ and $\Omega_\eps$. We define the invertible map
\[\Phi_{\sigma_0}\colon (x',x_n)\in V\mapsto\left(x',\dfrac{x_n-g(x')}{k_\eps(x')}\right)\in \tilde{V},\]
where $\tilde V=\Phi_{\sigma_0}(V)$. For simplicity's sake, when possible, we will drop the explicit dependence on the point $\sigma_0\in\partial\Omega$.
Notice that the map $\Phi$ indeed flattens the boundaries of $\Omega$ and $\Omega_\eps$, in the sense that
\[\Phi(\partial\Omega\cap V)=\set{y_n=0}\cap \tilde{V},\]
and \[ \Phi(\partial\Omega_\eps\cap V)=\set{y_n=\eps}\cap \tilde{V}.\]
For every $\delta>0$ we define the cube 
\[\tilde{Q}_\delta=\set{y\in\R^n| \abs{y_i}\le\delta,\,\text{for every }1\le i\le n}.\]
Up to choosing a smaller neighborhood $V$ of $\sigma_0$, we may assume that 
\[
\tilde{V}=\Phi(V)=\tilde{Q}_{\delta_0}
\]
for some $\delta_0=\delta_0(\sigma_0,\Omega)>0$. For every $0<\delta<\delta_0$ we define
\[Q_\delta=\Phi^{-1}(\tilde{Q}_\delta).\]
\begin{oss}[Estimates for $k_\eps$, $\Phi$, and $\Phi^{-1}$]
\label{rem: kepsestim}
     We claim that 
    there exists a positive constant $C=C(\Omega,\norma{h}_{C^{1,1}})$ such that 
    \begin{equation}
        \label{eq: kepsstime}
        \min k_\eps\ge \frac{1}{C}, \qquad\qquad  \norma{k_\eps}_{C^{1,1}}+\norma{\Phi}_{C^{1,1}}+\norma{\Phi^{-1}}_{C^{1,1}}\le C.
    \end{equation}
    Notice that, since $\Omega\subset\Omega_\eps$ for every $\eps>0$, then $k_\eps$ is non negative. Moreover, for every $x=(x',x_n)\in\partial\Omega_\eps\cap V $ 
\[
\eps k_\eps(x')\ge d(x,\partial\Omega)\ge\eps \min_{\Gamma_{d_0}} h,
\]
so that $k_\eps$ is strictly bounded from below uniformly in $\eps$. Similarly, $k_\eps$ is bounded in $C^{1,1}$ norm uniformly in $\eps$, indeed for every $x\in\partial\Omega_\eps\cap V$, we can write $x=(x', g(x')+\eps k_\eps(x'))$, and we have that
\[d(x,\partial\Omega_\eps)=0.\]
Differentiating with respect to $x'$ we get
\begin{equation}\label{kimpl}\nu_\eps' +\nu_\eps\cdot\mathbf{e}_n(\nabla g+\eps\nabla k_\eps)=0,\end{equation}
where $\nu_\eps'\in\R^{n-1}$ is the vector whose components are the first $n-1$ components of $\nu_\eps$. Using the fact that
\[\nu_\eps=\dfrac{\nu_0-\eps\nabla h}{\sqrt{1+\eps^2\abs{\nabla h}^2}},\]
and that in $V$
\[\nu_0=\left(-\dfrac{\nabla g}{\sqrt{1+\abs{\nabla g}^2}},\dfrac{1}{\sqrt{1+\abs{\nabla g}^2}}\right),\]
we can rewrite equation \eqref{kimpl} as
\begin{equation}
\label{nablak}
\nabla k_\eps=\dfrac{(\nabla h)'+(\nabla h\cdot\mathbf{e}_n) \nabla g}{(\nu_0-\eps\nabla h)\cdot\mathbf{e}_n}.
\end{equation}
Finally, the choice of $\eps_0$ in \autoref{lem: uniformrepr} (see \eqref{eq: Fxincreas}) ensures us that
\begin{equation}
\label{eq: niepsen}
    (\nu_0-\eps\nabla h)\cdot\mathbf{e}_n>\dfrac{1}{4},
\end{equation}
so that from equation \eqref{nablak} and \eqref{eq: niepsen}, \eqref{eq: kepsstime} follows.
\end{oss}

\begin{oss}[The equation in the flattened set]
\label{rem: flateq}
    Let $u_\eps$ be the solution to \eqref{eq: weakeq}, fix $\sigma_0\in\partial\Omega$, and let
    \[
    v(y)=u_\eps(\Phi^{-1}(y)),
    \]
    then \[v\in H^1\bigl(\set{y_n<\eps}\cap\tilde{V}\bigr)\cap H^2_{\loc}\Bigl(\bigl(\set{y_n<\eps}\cap\tilde{V}\bigr)\setminus\set{y_n=0}\Bigr)\] and, for all $\varphi\in H^1_0(\tilde{V})$, $v$ solves the equation
    \begin{equation}
    \label{eq: recteq}
    \begin{aligned}
        \int_{\{y_n< \eps\}\cap \tilde{V}} \eps(y_n) A_\eps \nabla v\cdot \nabla \varphi\, dy +\beta \int_{\{y_n=\eps\}\cap \tilde{V}}v\varphi J_\eps\, d\Hn = \int_{\{y_n<0\}\cap \tilde{V} }\tilde f_\eps \varphi \, dy,%
    \end{aligned}
    \end{equation}
    where
\[
    \eps(y_n)=\begin{cases}
        \eps &\text{if }y_n>0, \\
        1 &\text{if } y_n\le 0,
        \end{cases}
    \]
    \medskip
    \[
        A_\eps(y) = k_\eps(y') (D(\Phi^{-1})(y))^{-1}(D(\Phi^{-1})(y))^{-T} , \qquad\qquad \quad  \tilde{f}_\eps(y)=f(\Phi^{-1}(y))\,k_\eps(y'),
    \]    
    and 
    \[
        J_\eps(y)=\sqrt{1+\abs{\nabla g(y')+\eps \nabla k_\eps(y')}^2}.
    \]
    Notice that $A_\eps$ is elliptic and bounded, uniformly in $y$ and $\eps$. Moreover, using \eqref{eq: firstderbound} in \autoref{cor: firstorderestim}, we also get that there exists a positive constant $C=C(\Omega, h, \beta, f,\sigma_0)$ such that 
    \begin{equation}
    \label{eq: firstderboundrect}
    \int_{\set{y_n<\eps}\cap\tilde{V}}\eps(y_n)\abs{\nabla v}^2\, dy+\beta\int_{\{y_n=\eps\}\cap \tilde{V}}v^2\,d\Hn \le C. 
    \end{equation}
\end{oss}
\subsection{Step 2: \texorpdfstring{$H^2$}{} uniform estimates}
Since we aim to prove $H^2$ estimates with a local approach, we define the energy quantities $I_\delta$ and $\tilde{I}_\delta$ as follows: given a function 
\[
\varphi\in H^1\bigl(\set{y_n<\eps}\cap\tilde{V}\bigr)\cap H^2\Bigl(\bigl(\set{y_n<\eps}\cap\tilde{V}\bigr)\setminus\set{y_n=0}\Bigr)
\]
and $0<\delta<\delta_0$, we denote by
\begin{equation} 
\label{def:tildeI}
\tilde{I}_{\delta,\sigma_0}(\varphi)=\int_{\set{y_n<\eps}\cap\tilde{Q}_\delta}\eps(y_n)\abs{D^2 \varphi}^2\, dy+\beta\int_{\{y_n=\eps\}\cap \tilde{Q}_\delta}\abs{\nabla_{n-1}\varphi}^2\,d\Hn,
\end{equation}
where
\[
\nabla_{n-1} \varphi=\nabla \varphi-\dfrac{\partial \varphi}{\partial y_n} \mathbf{e}_n,\]
and, as in \autoref{rem: flateq},
\[
    \eps(t)=\begin{cases}
        \eps &\text{if }t>0, \\
        1 &\text{if } t\le 0.
        \end{cases}
\]
Analogously in $\Omega_\eps$, for every
\[
\varphi\in H^1\bigl(\Omega_\eps\cap V\bigr)\cap H^2\Bigr(\bigr(\Omega_\eps\cap V\bigl)\setminus\partial\Omega\Bigr),
\]
we let
\begin{equation} 
\label{def:I}
I_{\delta,\sigma_0}(\varphi)=\int_{\Omega_{\eps}\cap{Q}_\delta}\eps(d(x))\abs{D^2 \varphi}^2\, dx+\beta\int_{\partial\Omega_\eps\cap {Q}_\delta}\abs{\nabla^{\partial\Omega_\eps}\varphi}^2\,d\Hn,
\end{equation}
where we recall that
\[\nabla^{\partial\Omega_\eps}\varphi=\nabla\varphi-\left(\nabla \varphi\cdot\nu_\eps\right)\nu_\eps.\]
When possible, we will drop the dependence on $\sigma_0$. Uniform bounds for ${I}_\delta$ can be read as uniform bounds for $\tilde{I}_\delta$ and viceversa. Indeed, we have the following
\begin{lemma}
    Let $\Omega\subset\R^n$ be a bounded, open set with $C^{1,1}$ boundary, fix a positive function $h\in C^{1,1}(\Gamma_{d_0})$ such that $\eps\norma{h}_{C^{0,1}}\le\eps_0$ and $\nabla h\cdot \nu_0=0$. Then, for every $\sigma_0\in\partial\Omega$, there exists a positive constant $C=C(\Omega, \norma{h}_{C^{1,1}},\sigma_0)$ such that if 
    \[
     u \in H^1(Q_\delta)\cap H^2(Q_\delta\setminus\partial\Omega),
    \]
    and
    \[
        v(y)=u(\Phi^{-1}(y)),
    \]
    then, for every $0<\delta<\delta_0$,
      \begin{equation}
        \label{eq: I<Itilde}
         {I}_\delta(u)\le C\left(\tilde{I}_\delta(v)+\int_{{\{y_n<\eps\}}\cap\tilde{Q}_\delta}\eps(y_n)\abs{\nabla v}^2\,dy\right),
    \end{equation}
    and
      \begin{equation}
        \label{eq: Itilde<I}
         \tilde{I}_\delta(v)\le C\left({I}_\delta(u)+\int_{\Omega_\eps\cap Q_\delta}\eps(d(x))\abs{\nabla u}^2\,dx\right).
    \end{equation}
    \begin{proof}
        We start by evaluating the trace term in the definition of $\tilde I_\delta$. By means of the change of variables $y=\Phi(x)$, we have
        \begin{equation}
            \label{eq: L2horizgrad}
        \begin{split}
        \int_{{\{y_n=\eps\}}\cap\tilde{Q}_\delta}\abs{\nabla_{n-1}v}^2\,d\Hn(y)&=\sum_{i=1}^{n-1}\int_{\partial\Omega_\eps\cap V}(\nabla u\cdot w_i)^2\,J^{\partial\Omega_\eps}\Phi\,d\Hn(x),
        \end{split}
        \end{equation}
        where 
        \[
        w_i=\mathbf{e}_i+(\partial_i g+\eps \partial_i k_\eps)\,\mathbf{e}_n.
        \]
        The vectors $w_i$ (that, in general, could be non-orthogonal) form a basis for the tangent plane $T_{\sigma_0}\partial\Omega_\eps$. In particular, we have that there exists a positive constant $C=C(\Omega,h,\sigma_0)$ such that
        \begin{equation}
            \label{eq: tangentgrad}
        \abs{\nabla^{\partial\Omega_\eps}u}^2\le C\sum_{i=1}^{n-1}(\nabla u\cdot w_i)^2\le C^2\abs{\nabla^{\partial\Omega_\eps}u}^2.
        \end{equation}
        Therefore, using the uniform bounds \eqref{eq: kepsstime} we get that for some positive constant $C=C(\Omega,h)$
        \begin{equation}
        \label{eq: tangjacest}
            \frac{1}{C}\le J^{\partial\Omega_\eps}\Phi=\frac{1}{\sqrt{1+\abs{\nabla g+\eps\nabla k_\eps}^2}}\le C,
        \end{equation}
        and joining \eqref{eq: L2horizgrad}, \eqref{eq: tangentgrad}, and \eqref{eq: tangjacest}, we get
        \begin{equation}
        \label{eq: traceequiv}
        \int_{\partial\Omega_\eps\cap {Q}_\delta}\abs{\nabla^{\partial\Omega_\eps}\varphi}^2\,d\Hn\le C \int_{{\{y_n=\eps\}}\cap\tilde{Q}_\delta}\abs{\nabla_{n-1}v}^2\,d\Hn\le C^2 \int_{\partial\Omega_\eps\cap {Q}_\delta}\abs{\nabla^{\partial\Omega_\eps}\varphi}^2\,d\Hn.
        \end{equation}

        \bigskip
        For what concerns the second order term, we evaluate for a.e. $y\in \{y_n<\eps\}\cap\tilde{V}$,
        \[
            \big(D^2 v(y)\big)_{ij}=\sum_{k=1}^{n}\frac{\partial^2 (\Phi^{-1})_k}{\partial y_i\,\partial y_j}\frac{\partial u}{\partial x_k}+\sum_{k,l=1}^{n}\frac{\partial (\Phi^{-1})_k}{\partial y_i}\frac{\partial^2 u}{\partial x_k\,\partial x_l}\frac{\partial (\Phi^{-1})_l}{\partial y_j},
        \]
        so that, for some positive constant $C=C(n,\norma{\Phi^{-1}}_{C^{1,1}})$,
        \begin{equation}
        \label{eq: D2vD2u}
        \abs{D^2 v\big(\Phi(x)\big)}^2\le C\left(\abs{\nabla u(x)}^2+\abs{D^2 u(x)}^2\right).
        \end{equation}
        In a similar way we have that, for a.e. $x\in \Omega_\eps\cap V$, and for some positive constant $C=C(n,\norma{\Phi}_{C^{1,1}})$,
        \begin{equation}
        \label{eq: D2uD2v}
        \abs{D^2 u\big(\Phi^{-1}(y)\big)}^2\le C\left(\abs{\nabla v(y)}^2+\abs{D^2 v(y)}^2\right).
        \end{equation}
        Using \eqref{eq: D2vD2u} and the uniform bounds on $k_\eps$, we get
         \begin{equation}
         \label{eq: D2v<D2u}
        \begin{split}
        \int_{{\{y_n<\eps\}}\cap\tilde{Q}_\delta}\eps(y_n)\abs{D^2 v}^2\,dy&=\int_{\Omega_\eps\cap Q_\delta}\frac{\eps(d(x))}{k_\eps}\abs{D^2v\big(\Phi(x)\big)}^2\,dx \\[7 pt]
        &\le C\int_{\Omega_\eps\cap Q_\delta}\eps(d(x))\big(\abs{D^2u}^2+\abs{\nabla u}^2\big)\,d\Hn(x). 
        \end{split}
        \end{equation}
        Analogously, using \eqref{eq: D2uD2v},
        \begin{equation}
        \label{eq: D2u<D2v}
        \begin{split}
        \int_{\Omega_\eps\cap Q_\delta}\eps(d(x))\abs{D^2u}^2\,dx\le C\int_{{\{y_n<\eps\}}\cap\tilde{Q}_\delta}\eps(y_n)\left(\abs{D^2 v}^2+\abs{\nabla v}^2\right)\,d\Hn.
        \end{split}
        \end{equation}
        The result now follows by joining \eqref{eq: traceequiv}, \eqref{eq: D2u<D2v}, and \eqref{eq: D2v<D2u}.
    \end{proof}
    \end{lemma}

    \begin{lemma}
    \label{lem: energyestim}
        Let $\Omega\subset\R^n$ be a bounded, open set with $C^{1,1}$ boundary, fix a positive function $h\in C^{1,1}(\Gamma_{d_0})$ such that $\eps\norma{h}_{C^{0,1}}\le\eps_0$ and $\nabla h\cdot \nu_0=0$. If $\sigma_0\in\partial\Omega$, and $v$ is as in \autoref{rem: flateq}, then 
        \[
            v\in H^2\left(\set{y_n<\eps}\cap \tilde{Q}_{\delta_0/2}\setminus\set{y_n=0}\right),
        \]
    and there exists a positive constant $C=C(\Omega, \norma{h}_{C^{1,1}}, \beta, f, \sigma_0)$ such that
    \begin{equation}
        \label{eq: ideltaest}
         \tilde{I}_{\delta_0/2}(v)\le C.
         \end{equation}
    \end{lemma}
           \begin{proof}
            Let $\xi\in C^\infty_c(\tilde{Q}_{\delta_0})$ be a non negative function with $\xi\equiv1$ in $\tilde{Q}_{\delta_0/2}$. For $\abs{\eta}$ small enough, we have 
            \[
            \supp\xi+\eta\mathbf{e}_i \ssubset \tilde{Q}_{\delta_0}
            \]
            for every $k=1,\dots,n-1$. In the following, for any $L^2$ function $\psi$, we denote by  
            \[
            \Delta_k^\eta \psi(x)=\frac{\psi(x+\eta\mathbf{e}_k)-\psi(x)}{\eta}
            \]
            the difference quotients (see for instance \cite[\S 5.8.2]{evans}). We recall that for any couple of functions $\psi_1$ and $\psi_2$ we have
            \[
                \Delta_k^\eta (\psi_1 \psi_2)(x)=\Delta_k^\eta\psi_1(x)\psi_2(x)+\psi_1(x+\eta\mathbf{e}_k)\Delta_k^\eta\psi_2(x),
            \]
            and if $\psi_1$ and $\psi_2$ are measurable and
            \[            (\supp\psi_1\cap\supp\psi_2)\pm\eta\mathbf{e}_k \ssubset \tilde{Q}_\delta,
            \]
            then for every $k=1,\dots,n-1$, it holds
            \begin{equation}
            \label{eq: intparts}
                \int_{\tilde{Q}_{\delta_0}}\psi_1(y)\Delta_k^{\eta}\psi_2(y)\,dy=-\int_{\tilde{Q}_{\delta_0}} \Delta_k^{-\eta}\psi_1(y)\: \psi_2(y)\,dy.
            \end{equation}
            Moreover, we recall that if $\psi\in H^1({\{y_n<\eps\}\cap\tilde{Q}_{\delta_0}})$, then for every $U\ssubset{\tilde{Q}_{\delta_0}}$ we have that %
            \begin{equation}
            \label{eq: dqestimL2}
            \int_U(\Delta_k^\eta \psi)^2\,dy\le \int_{\tilde{Q}_{\delta_0}}\abs{\nabla \psi}^2\,dy,
            \end{equation}
            for every $\abs{\eta}<d(U,\partial \tilde{Q}_{\delta_0})$.
            We aim to prove some uniform $L^2$ estimates for $\Delta_k^\eta(\nabla v)$, which will imply weak differentiability and uniform $L^2$ estimates for $\partial_k\nabla v$.
            
            Given the equation \eqref{eq: recteq}, we use as a test function
            \[
                \varphi=-\Delta_k^{-\eta}(\xi^2 \Delta_k^\eta v).
            \]
For small $\abs{\eta}$, the function $\varphi$ is admissible. Here we recall the weak equation
\[
 \underbrace{\int_{\{y_n< \eps\}\cap \tilde{V}} \eps(y_n) A_\eps \nabla v\cdot \nabla \varphi\, dy}_{\mathcal{I}_1} +\underbrace{\beta \int_{\{y_n=\eps\}\cap \tilde{V}}v\varphi J_\eps\, d\Hn}_{\mathcal{I}_2} = \underbrace{\int_{\{y_n<0\}\cap \tilde{V} }\tilde f_\eps \varphi \, dy}_{\mathcal{I}_3}.
\]
We now estimate the integrals separately. In the following, we use the Einstein notation on repeated indices, where $i,j=1,\dots,n$, and for simplicity we drop the dependence on $\eps$ in $A_\eps$, whose components are denoted by $a_{ij}$. We recall that $a_{ij}=a_{ij}(y)$ and $J_\eps=J_\eps(y)$. When necessary, we will also write $a_{ij}^\eta:=a_{ij}(y+\eta \mathbf{e}_k)$ and $J_\eps^\eta:=J_\eps(y+\eta\mathbf{e_k})$.

\medskip
Let us start with the higher-order term $\mathcal{I}_1$: by direct computation, we get
            \[
            \begin{split}
            A\nabla v\cdot \nabla\varphi&=-a_{ij}\:\partial_j v\:\Delta_k^{-\eta}\,\partial_i(\xi^2 \Delta_k^\eta v) \\[5 pt]
            &=-a_{ij}\,\partial_j v\,\left[\Delta_k^{-\eta}(\partial_i(\xi^2)\:\Delta_k^\eta v)+\Delta_k^{-\eta}(\xi^2 \Delta_k^\eta (\partial_i v))\right] 
            \end{split}
            \]
            Multiplying by $\eps(y_n)$, integrating over $\{y_n<\eps\}\cap\tilde{Q}_{\delta_0}$, and using property \eqref{eq: intparts}, we get (notice that $\eps(y_n)$ is constant along the direction $\mathbf{e}_k$)
            \begin{equation}
                \label{A}
            \begin{split}
                \mathcal{I}_1&=\int_{\{y_n<\eps\}\cap\tilde{Q}_{\delta_0}}\eps(y_n)\Delta_k^\eta(a_{ij}\,\partial_j v) \,\partial_i(\xi^2) \Delta_k^\eta v\,dy \\[7 pt]
                &\hphantom{=\,}+\int_{\{y_n<\eps\}\cap\tilde{Q}_{\delta_0}}\eps(y_n)\Delta_k^\eta(a_{ij}\partial_j v)\,\xi^2 \,\Delta_k^\eta (\partial_i v)\, dy\\[7 pt]
                &=\int_{\{y_n<\eps\}\cap\tilde{Q}_{\delta_0}}\eps(y_n)\xi^2 \,a_{ij}^\eta\, \Delta_k^\eta(\partial_j v)\,\Delta_k^\eta (\partial_i v)\, dy + R_1 \\[7 pt]
                & \ge C_1\int_{\{y_n<\eps\}\cap\tilde{Q}_{\delta_0}}\eps(y_n)\xi^2 \,\sum_{j=1}^n\big(\Delta_k^\eta(\partial_j v)\big)^2\,dy-\abs{R_1},
            \end{split}
            \end{equation}
            where $C_1$ is the ellipticity constant of $A$ estimated uniformly in \autoref{rem: flateq}, and
            \[
            \begin{split}
                R_1&=2\int_{\{y_n<\eps\}\cap\tilde{Q}_{\delta_0}}\eps(y_n)\left(\Delta_k^\eta (a_{ij})\partial_j v+a_{ij}^\eta\,\Delta_k^\eta(\partial_j v)\right) \, \xi\,\partial_i\xi \,\Delta_k^\eta v\,dy \\[7 pt]
                &\hphantom{=\,}+\int_{\{y_n<\eps\}\cap\tilde{Q}_{\delta_0}}\eps(y_n)\Delta_k^\eta(a_{ij})\partial_j v\,\xi^2 \,\Delta_k^\eta (\partial_j v)\, dy. 
            \end{split}
            \]
Using Young's inequality
\[
ab\le \dfrac{\lambda}{2} a^2+\dfrac{1}{2\lambda} b^2 \qquad\quad \forall a,b\ge 0,
\]
with a suitable choice of $\lambda>0$, and using the uniform bounds on $\norma{\xi}_{C^1}$ and $\norma{A_\eps}_{C^1}$ (see \autoref{rem: flateq} and \autoref{rem: kepsestim}), we may obtain a positive constant $C=C(A,\xi)$ such that the following estimate holds
\[
\begin{split}
\abs{R_1}&\le  \frac{C_1}{2}\int_{\{y_n<\eps\}\cap\tilde{Q}_{\delta_0}}\eps(y_n)\xi^2 \,\sum_{j=1}^n\big(\Delta_k^\eta(\partial_j v)\big)^2\,dy\\[7 pt] 
&\hphantom{= }+C\int_{\{y_n<\eps\}\cap\,\supp\xi} \eps(y_n)\left((\Delta_k^\eta v)^2+\sum_{j=1}^n(\partial_j v)^2\right)\,dy.
\end{split}
\]
This estimate, joint with \eqref{A}, implies
\begin{equation}
\label{Ainf}
\mathcal{I}_1\ge \frac{C_1}{2}\int_{\{y_n<\eps\}\cap\tilde{Q}_{\delta_0}}\eps(y_n)\xi^2 \,\sum_{j=1}^n\big(\Delta_k^\eta(\partial_j v)\big)^2\,dy-C \int_{\{y_n<\eps\}\cap\tilde{Q}_{\delta_0}} \eps(y_n)\abs{\nabla v}^2\,dy.
\end{equation}

\bigskip
Similarly, let us work with the boundary terms $\mathcal{I}_2$, where we have 
\[
    v\varphi J_\eps=-v J_\eps \Delta_k^{-\eta}(\xi^2\Delta_k^\eta v).
\]
Since $k\ne n$, we can use an analogous of property \eqref{eq: intparts} also on $\set{y_n=\eps}\cap\tilde{Q}_{\delta_0}$, so that
\begin{equation}
    \label{j}
\begin{split}
    \mathcal{I}_2&=\beta\int_{\set{y_n=\eps}\cap\tilde{Q}_{\delta_0}}\xi^2 \Delta_k^\eta(v J_\eps)\Delta_k^\eta v\,d\Hn \\[7 pt]
    &=\beta\int_{\set{y_n=\eps}\cap\tilde{Q}_{\delta_0}}\xi^2 (\Delta_k^\eta v)^2 J_\eps^\eta\,d\Hn+R_2 \\[7 pt]
    &\ge \beta C_2\int_{\set{y_n=\eps}\cap\tilde{Q}_{\delta_0}}\xi^2 (\Delta_k^\eta v)^2\,d\Hn-\abs{R_2},
\end{split}
\end{equation}
where $C_2=\inf_{Q_\delta} J_\eps$ uniformly estimated in \autoref{rem: flateq}, and
\[
 R_2=\beta\int_{\set{y_n=\eps}\cap\tilde{Q}_{\delta_0}}\xi^2 \Delta_k^\eta(J_\eps)v \Delta_k^\eta v\,d\Hn 
\]
As done for $\mathcal{I}_1$, using Young's inequality and the uniform bounds on 
$\norma{J_\eps}_{C^1}$ (see \autoref{rem: flateq} and \autoref{rem: kepsestim}), we have that for some positive constant $C=C(J_\eps,\xi)$,
\[
    \abs{R_2} \le \frac{\beta C_2}{2}\int_{\set{y_n=\eps}\cap\tilde{Q}_{\delta_0}}\xi^2 (\Delta_k^\eta v)^2\,d\Hn + C\beta\int_{\set{y_n=\eps}\cap\tilde{Q}_{\delta_0}} v^2\,d\Hn 
\]      
This estimate, joint with \eqref{j} ensures
\begin{equation}
\label{Jinf}
\mathcal{I}_2\ge \frac{\beta C_2}{2}\int_{\set{y_n=\eps}\cap\tilde{Q}_{\delta_0}}\xi^2 (\Delta_k^\eta v)^2\,d\Hn-C\beta\int_{\set{y_n=\eps}\cap\tilde{Q}_{\delta_0}} v^2\,d\Hn.
\end{equation}

\bigskip
Finally, we estimate the source term $\mathcal{I}_3$, recalling that
\[
   \tilde{f}_\eps\varphi =-\tilde{f}_\eps\Delta_k^{-\eta}(\xi^2 \Delta_k^\eta v).
\]
Using Young's inequality and property \eqref{eq: dqestimL2} for $\psi=\xi^2 \Delta_k^\eta v$ (notice that $\nabla \Delta_k^\eta v=\Delta_k^\eta(\nabla v)$), then for a suitable positive constant $C=C(\tilde{f},\xi,\delta_0)$,
        \begin{equation}\label{f}  
            \begin{split}
                \mathcal{I}_3 &\le \frac{C_1}{4}\int_{\set{y_n<0}\cap\tilde{Q}_{\delta_0}} \xi^2 \sum_{j=1}^n (\Delta_k^\eta \partial_j v)^2\,dy \\[7 pt] 
                &\hphantom{=}+ C\int_{\set{y_n<0}\cap\tilde{Q}_{\delta_0}}\tilde{f}_\eps^2\,dy+ C\int_{\set{y_n<0}\cap\tilde{Q}_{\delta_0}}\abs{\nabla v}^2\,dy \\[7 pt]
                &\le \frac{C_1}{4}\int_{\set{y_n<\eps}\cap\tilde{Q}_{\delta_0}} \eps(y_n)\xi^2 \sum_{j=1}^n (\Delta_k^\eta \partial_j v)^2\,dy\\[7 pt] 
                &\hphantom{=}+ C\int_{\set{y_n<0}\cap\tilde{Q}_{\delta_0}}\tilde{f}_\eps^2\,dy+ C\int_{\set{y_n<\eps}\cap\tilde{Q}_{\delta_0}}\eps(y_n)\abs{\nabla v}^2\,dy.
            \end{split}
        \end{equation}
        
We can now turn back to the equation
\[
\mathcal{I}_1+\mathcal{I}_2=\mathcal{I}_3
\]
Joining \eqref{Ainf}, \eqref{Jinf}, \eqref{f}, and using the fact that $\xi=1$ in $\tilde{Q}_{\delta_0/2}$, then, for every $k=1,\dots,n-1$,
\begin{equation}
    \label{D^}
    \begin{split}
                \mathcal{I}_4&:=\frac{C_1}{4}\int_{\set{y_n<\eps}\cap\tilde{Q}_{\delta_0/2}} \eps(y_n)\sum_{j=1}^n (\Delta_k^\eta \partial_j v)^2\,dy\\[7 pt] 
                &\hphantom{=}+\frac{\beta C_2}{2}\int_{\set{y_n=\eps}\cap\tilde{Q}_{\delta_0/2}} (\Delta_k^\eta v)^2\,d\Hn \\[7 pt] 
                &\le C \Bigg(\int_{\set{y_n<\eps}\cap\tilde{Q}_{\delta_0}} \eps(y_n)\abs{\nabla v}^2\,dy\\[7 pt] 
                &\hphantom{= C\Bigg(}+\beta\int_{\set{y_n=\eps}\cap\tilde{Q}_{\delta_0}}v^2\,d\Hn+\int_{\set{y_n<0}\cap\tilde{Q}_{\delta_0}} \tilde{f}_\eps^2\,dy\Bigg).
    \end{split}
\end{equation}
Since we have assumed the uniform $H^1$ estimates \eqref{eq: firstderbound}, and we have uniform estimates for $\tilde{f}_\eps$, computed in \autoref{rem: flateq}, we get that for some positive constant $C=C(\delta_0,\Omega,\xi)$
\[
        \mathcal{I}_4\le C,
\]
which implies that for every $k=1,\dots,n-1$, we have $\partial_k v\in H^1(\{y_n<\eps\}\cap \tilde{Q}_{\delta_0/2})$ and
\[\begin{split}\int_{\set{y_n<\eps}\cap \tilde{Q}_{\delta_0/2}} \eps(y_n) \sum_{j=1}^n(\partial_k\partial_j v)^2\,dy&\le C \Bigg(\int_{\set{y_n<\eps}\cap\tilde{Q}_{\delta_0}} \eps(y_n)\abs{\nabla v}^2\,dy\\[7 pt] 
                &\hphantom{= C\Bigg(}+\beta\int_{\set{y_n=\eps}\cap\tilde{Q}_{\delta_0}}v^2\,d\Hn+\int_{\set{y_n<0}\cap\tilde{Q}_{\delta_0}} \tilde{f}_\eps^2\,dy\Bigg). \end{split}\]

It only remains to get a uniform estimate for $\partial^2_{nn}v$. We notice that since $v$ solves \eqref{eq: recteq}, then almost everywhere we have
\[
     -\sum_{\substack{1\le i,j\le n \\ (i,j)\ne(n,n)}}\eps(y_n)\partial_i(a_{ij}\partial_{j}v)-\eps(y_n)\partial_n a_{nn}\partial_n v-\eps(y_n)a_{nn}\partial^2_{nn}v=\tilde{f}\chi_{\{y_n<0\}}.
    \]
     In particular, the fact that $a_{nn}=k_\eps$ is uniformly bounded from below gives estimates for $\partial^2_{nn}v$ in terms of $A_\eps$, $\tilde{f}$, and the other derivatives of $v$. Therefore, so that, for every $0<\delta<\delta_0$, we can find a positive constant $C=C(\Omega,h,f)$ such that
            \begin{equation}\label{Vnn}
            \begin{split}
                \int_{\set{y_n<\eps}\cap\tilde{Q}_{\delta}}\eps(y_n)\abs{\partial_{nn}^2 v}^2\le C\Bigg(&\int_{\set{y_n<\eps}\cap\tilde{Q}_{\delta}}\eps(y_n)\sum_{k=1}^{n-1}{\abs{\nabla \partial_k v}}^2\,dy \\[7pt] 
                &+\int_{\set{y_n<\eps}\cap\tilde{Q}_{\delta}}\eps(y_n)\abs{\nabla v}^2\,d\Hn + 1\Bigg).
            \end{split}
            \end{equation}
  Finally, \eqref{D^}, joint with \eqref{Vnn} and the bound \eqref{eq: firstderboundrect}, gives \eqref{eq: ideltaest}.
           
        \end{proof}

\begin{oss}
\label{rem: equivIde}
        Let $\sigma_0\in\partial\Omega$, by the previous lemma we have that $u_\eps\in H^2(Q_{\delta_0/2}\setminus\partial\Omega)$. Moreover, putting together estimates \eqref{eq: I<Itilde} and \eqref{eq: Itilde<I} and the bounds on the first derivative (\eqref{eq: firstderbound} and \eqref{eq: firstderboundrect}), we have
        \begin{equation}
            \label{eq: IdeltatoF}
            I_{\delta}(u_\eps)\le C(1+\tilde{I}_\delta(v_\eps))\le C^2(1+I_\delta(u_\eps)).
        \end{equation}
    \end{oss}

We can finally prove \autoref{teor: C11energy}.

\begin{proof}[Proof of \autoref{teor: C11energy}]
We recall that we have defined $\tilde{I}_{\delta,\sigma}$ and $I_{\delta,\sigma}$ in \eqref{def:tildeI} and \eqref{def:I} respectively. For every $\sigma_0\in\partial\Omega$, using \eqref{eq: ideltaest} from \autoref{lem: energyestim}, an and \eqref{eq: IdeltatoF} from \autoref{rem: equivIde}, we have that there exists a constant $C=C(\Omega,h,\sigma_0,\delta_0)$ such that
\begin{equation}\label{est}
I_{\delta_0/2,\sigma_0}(u_\eps)\le C.
\end{equation}
From the boundedness of $\Omega$ there exist $\sigma_1,\dots,\sigma_m\in\partial\Omega$ and associated $\delta_i=\delta(\sigma_i,\Omega)$ for which estimates of the type \eqref{est} hold and such that, choosing $\eps_0=\min\{\eps_0(\sigma_1),\dots,\eps_0(\sigma_m)\}>0$, we have
\[
\Sigma_\eps\subset\bigcup_{i=1}^m \Phi_{\sigma_i}^{-1}(\tilde{Q}_{\delta_i/2})=V_0.
\]
Let $U$ be an open set such that $\overline{U}\subset\Omega$ and $\Omega_\eps\subset U\cup V_0$ for every $\eps$ such that $\eps\norma{h}_{C^{0,1}}<\eps_0$. Also in $U$ we may get an estimate analogous to \eqref{est}. Indeed, by standard elliptic regularity and by estimate \eqref{eq: L2bound}, we have that
\begin{equation}\label{estint}
\int_U \abs{D^2 u_\eps}^2\,dx\le C\left(\int_\Omega f^2\,dx+\int_\Omega u_\eps^2\right)\le C.
\end{equation}
Let 
\[I(u_\eps)=\int_{\Omega} \abs{D^2 u_\eps}^2\,dx+\eps\int_{\Sigma_\eps}\abs{D^2 u_\eps}^2\,dx+\beta\int_{\partial\Omega_\eps} \abs{\nabla^{\partial\Omega_\eps} u_\eps}^2\,d\Hn,\]
summing from $i=1$ to $m$  the estimates of the type \eqref{est} and \eqref{estint}, we have that there exists $C=C(\Omega,h,f)$ such that
\[I(u_\eps)\le C,\]
and the assertion is proven.
\end{proof}

\section*{Acknowledgements}
The authors are members of Gruppo Nazionale per l’Analisi Matematica, la Probabilità e le loro Applicazioni
(GNAMPA) of Istituto Nazionale di Alta Matematica (INdAM). 
The authors were partially supported by the project GNAMPA 2023: "Symmetry and asymmetry in PDEs", CUP\_E53C22001930001, and by the 2024 project GNAMPA 2024: "Modelli PDE-ODE nonlineari e proprieta' di PDE su domini standard e non-standard", CUP\_E53C23001670001.

 \bigskip

 \newpage 
\printbibliography[heading=bibintoc]
\Addresses
\end{document}